\documentclass[reqno]{amsart}

\usepackage{amsfonts,amsthm,amsmath,amssymb,amscd}
\usepackage{mathrsfs}
\usepackage{mathtools}
\usepackage{graphics}
\usepackage{color}
\usepackage{indentfirst}
\usepackage{cite}
\usepackage{latexsym}
\usepackage[dvips]{epsfig}

\usepackage{hyperref}

\newtheorem{Theorem}{Theorem}
\newtheorem{Corollary}{Corollary}
\newtheorem{Lemma}{Lemma}

\newtheorem{Definition}{Definition}
\newtheorem{Remark}{Remark}

\def\mcD{\mathcal{D}}
\def\mcE{\mathcal{E}}

\def\mfD{\mathfrak{D}}

\def\f{\vec{f}}
\def\e{e}
\def\R{\mathbb{R}}

\def\T{\mathbb{T}}

\def\Rp{\mathbb{R}_+}
\def\bRp{\bar{\mathbb{R}}_+}

\def\barDt{\frac{\overline{D}}{Dt}}
\def\bbarDt{\frac{\overline{\overline{D}}}{Dt}}

\def\eps{\varepsilon}
\newcommand{\supp}{\text{supp }}
\newcommand{\ptl}{\partial}

\numberwithin{equation}{section}

\begin{document}

\title[H\"{o}lder Solutions to 3D Prandtl System]{H\"{o}lder Continuous Solutions to the Three-dimensional Prandtl System}
\author[]{Tianwen Luo}
\address{Yau Mathematical Sciences Center, Tsinghua University, China}
\email{twluo@mail.tsinghua.edu.cn}
\author[ ]{Zhouping Xin}
\address{The Institute of Mathematical Sciences and department of
	mathematics, The Chinese University of Hong Kong, Hong Kong}
\email{zpxin@ims.cuhk.edu.hk}
\date{}
\begin{abstract}
	Adapting the convex integration technique introduced in \cite{dLSz5} and subsequently developed in \cite{IsettVicol,IsettOh16,Buckmaster2013transporting},  we construct  H\"{o}lder continuous weak solutions to the three dimensional Prandtl system and some other models with vertical viscosity.
\end{abstract}

\maketitle

\section{Introduction and Main Results}
In this paper, we will consider the three-dimensional Prandtl system, which is given by
\begin{align}
\begin{cases}\label{eq:Prandtl}
\ptl_t u + (u \cdot \nabla_x ) u  + v \ptl_y u + \nabla_x P = \ptl_{yy}^2 u,\\
\nabla_x \cdot u + \ptl_y v = 0,\\
u|_{t=0} = u_0, \quad
(u,v)|_{y=0}=0, \quad
\lim_{y\to+\infty}u = U.
\end{cases}
\end{align}
Here $x=(x_1,x_2) \in \mathbb{T}^2$ and $y \in \Rp$ denote the tangential and the vertical components of the space variable, respectively; $\nabla_x=(\ptl_{x_1},\ptl_{x_2})$ denotes the tangential gradient;
$u=(u^1,u^2)(t,x,y)$ and $v=v(t,x,y)$ denote the tangential and the vertical velocities;   $U(t,x)$ and $P(t,x)$ denote the tangential velocity and the pressure on the boundary $\{y=+\infty\}$ of the outer Euler flow, respectively, which satisfy 
\begin{align*}
\ptl_t U + (U\cdot \nabla_x ) U  + \nabla_x P = 0.
\end{align*}

The motion of a fluid as governed by the incompressible Navier-Stokes equations, may be well approximated by smooth inviscid flows in the limit of large Reynold numbers, except near the physical boundary where the effect of viscosities plays a significant role. There a thin layer forms in which the tangential velocity of the flow drops rapidly to zero at the boundary (no-slip condition). This layer is called the boundary layer, and of thickness $\sqrt{\nu}$ with $\nu$ being the viscosity coefficient. The theory of boundary layers was first proposed by Prandtl in 1904.  In Prandtl's theory, the flow outside the layer can be described approximately by the Euler equations, however, within the boundary layer, the flow is governed by a degenerate mixed-type system appropriately reduced from the Navier-Stokes equations, known as  the Prandtl system.

There has been a lot of mathematical literature on the Prandtl system with a focus on the two space-dimension case. The local-wellposedness and the rigorous justification of the viscous limit as the superposition of the Prandtl and Euler equations has been proved for analytic functions in \cite{SC98}. Recently these results are obtained for Gevrey classes in \cite{GM15,GMM16} and for Sobolev data with vorticity away from the boundary in \cite{Maekawa14}.  On the other hand, under the  monotonicity assumption on the tangential velocity of the data, the local well-posedness of classical solutions was obtained  by Oleinik and her co-workers\cite{OS99} using the Crocco transfrom and recently obtained  using energy methods in \cite{AWXY14,MW12}. Furthermore, the  global well-posedness of weak and smooth solutions was proved in \cite{XZ04,XZZ},
assuming further a favorable pressure condition. The finite-time blow-up of smooth solutions was obtained in \cite{EE97} and results of instability without the monotonicity assumption in  \cite{Grenier00,GN11}. However, despite a lot of progress, the validity of Prandtl's theory in the general case remains an open problem.

There are fewer results on the three dimensional Prandtl system. The viscous limit is obtained in  the analytic framework \cite{SC98}, and the recent work \cite{FTZ16} for Sobolev data  with vorticity away from the boundary using energy methods. The three dimensional Prandtl system appears to be quite challenging, mainly due to the possible onsets of secondary flows.

In the three dimensional boundary layers, the flows near  the boundary may develop motions transverse to the outer flow $U$, which are called secondary flows in the literature. It tends to occur when the pressure gradient does not align with the direction of the outer flow $U$; see \cite{Moore56} for discussions. This poses great challenges to the analysis.  Recently,
Liu, Wang and Yang proved the local existence of solutions to the three dimensional Prandtl systems with a special structure in \cite{LWY14a}, and also obtained ill-posedness results when the structural assumptions were violated in \cite{LWY14b}.
 It should be noted that the special structure in \cite{LWY14a} corresponds to the non-transversal of the tangential velocity and thus excludes secondary flows.

The main purpose of this paper is to obtain weak solutions to the system \eqref{eq:Prandtl} with tangential velocities transverse to the outflow $U$,  indicating the onsets of  secondary flows.

The weak solutions to the initial-boundary value problem \eqref{eq:Prandtl} is defined as follows.

\begin{Definition}
	A function $(u,v) \in C^0(\bRp \times \T^2 \times \bRp)$  is said to be a weak solution to the problem \eqref{eq:Prandtl}, if it solves the equations in the sense of distribution and satisfies the boundary conditions in \eqref{eq:Prandtl}.
\end{Definition}

The main results  can be stated as follows. Given $\alpha, \beta \in (0,1)$, for a continuous function $f$ defined on a closed space-time domain $\Omega$, let $[f]_{{\alpha,\beta}(\Omega)}$ denotes
\begin{align*}
[f]_{{\alpha,\beta}(\Omega)} =  \sup_{(t,x,y),(t',x',y') \in \Omega}\frac{|f(t,x,y)-f(t',x',y')|}{|t-t'|^{\alpha}+|x-x'|^{\alpha}+|y-y'|^{\beta}}.
\end{align*}
For any set $E \subset \Rp  \times \Rp$ and any positive numbers $\rho, \rho'$, we denote
\begin{align}
N(E;\rho, \rho') = \{ (t,y): |t-t_0| < \rho, |y-y_0| < \rho', \text{ for some } (t_0,y_0) \in E  \}. \label{def:N-t,y}
\end{align}
\begin{Theorem}\label{Thm:Prandtl}
	Suppose that $(u_C,v_C)$ is a classical solution to the system \eqref{eq:Prandtl} and  $(\underline{u},\underline{v})$  is a smooth   perturbation of $(u_C,v_C)$   such that the difference $(\underline{u} - u_C,\underline{v}-v_C) \in C_c^\infty(\Rp \times \T^2 \times \Rp)$ has compact support and satisfies
	\begin{align}
	\nabla_x \cdot \underline{u} + \ptl_y \underline{v} &= 0, \quad \text{ in } \Rp \times \T^2 \times \Rp, \label{eq:div-bu-bv}\\
	\int_{\T^2} (\underline{u}-u_C)(t,x,y)  dx &= 0, \quad \text{ for any } (t,y) \in \Rp \times \Rp. \label{u-u_S}
	\end{align}
	Let $\rho>0$ be a given positive number such that
	\begin{align*}
	\overline{N(\text{supp}_{t,y}(\underline{u} - u_C,\underline{v}-v_C);\rho,\rho^{1/2})}
	 \subset \Rp \times \Rp,
	\end{align*}
	where $\text{supp}_{t,y}$ denotes the projection of the support to $\Rp \times \Rp$.
	Then there exists a sequence of H\"{o}lder continuous weak solutions $\{(u_k,v_k)\}_{k=1}^{\infty}$ to the system \eqref{eq:Prandtl} and  a sequence of positive numbers $\{C_k\}$ satisfying the estimates
	\begin{equation}
	[u_k - \underline{u},v_k - \underline{v}]_{\frac{1}{21} - \epsilon,\frac{1}{10} - \epsilon} < C_k, \label{ineq:u-u_S}
	\end{equation}
	and
	\begin{align}
	\text{supp}_{t,y} (u_k - \underline{u},v_k - \underline{v}) \subset N(\text{supp}_{t,y}(\underline{u} - u_C,\underline{v}-v_C);\rho,\rho^{1/2}). \label{supp-diff-u_k-v_k}
	\end{align}
	Furthermore, $(u_k,v_k) \rightharpoonup (\underline{u},\underline{v})$ in the weak-$*$ topology on $L^\infty(\Rp \times \T^2 \times \Rp)$.

\end{Theorem}

Given a uniform outflow $U=\text{const}$, for any initial data $u_S^0(y)$ depending only on the vertical variable $y$, it is well-known that the system \eqref{eq:Prandtl} admits a shear flow solution $(u_S(t,y),0)$,
which is the unique solution to the following heat equation:
\begin{align*}
\begin{cases}
\ptl_t u_S = \ptl_{yy}^2 u_S, \quad \text{ in } \Rp \times \Rp,\\
u_S|_{y=0}=0, \quad \lim\limits_{y\to +\infty} u_S(t,y)=U,\\
u_S|_{t=0}=u_S^0(y).
\end{cases}
\end{align*}
Applying  Theorem \ref{Thm:Prandtl} to the shear flow $(u_S(t,y),0)$, we obtain the following results.
\begin{Corollary} \label{Corollary:ShearFlow}
	There exists a H\"{o}lder continuous weak solution $(u,v)$ satisfying the same initial-boundary conditions as the shear flow $u_S$. Furthermore, the tangential velocity $u$ is not monotonic in  $y$ and the flow $(u,v)$ is transverse to the outflow $(U,0)$ at some point $(t_0,x_0,y_0) \in \Rp \times \T^2 \times \Rp$.
\end{Corollary}

\begin{Remark}
	The construction of the  H\"{o}lder continuous weak solutions to the three dimensional Prandtl system exploits essentially the degree of freedom of the multi-dimensional tangential velocity space. It seems that similar constructions would not work directly for the two-dimensional Prandtl system.
\end{Remark}

The constructions in the proof of Theorem \ref{Thm:Prandtl} can also be adapted to some other models with vertical viscosities. In particular, consider the system
\begin{align}
\begin{cases}\label{eq:Euler-partial}
\ptl_t u + (u \cdot \nabla ) u  + \nabla P = \ptl^2_{yy} u,\\
\nabla \cdot u = 0,\\
u|_{t=0}=u_0,
\end{cases}
\end{align}
where $(x_1,x_2,y) \in \mathbb{T}^3, \nabla = (\ptl_{x_1},\ptl_{x_2},\ptl_{y}), u=(u^1,u^2,u^3)(t,x,y)$ and $P=P(t,x,y)$ denote the space variable, the spatial gradient, the velocity and the pressure of the flow, respectively.
\begin{Theorem}\label{Thm:Euler-partial}
	There exists a non-trivial H\"{o}lder continuous weak solutions $u$ to the system \eqref{eq:Euler-partial} which is supported in a compact time interval, with
	\begin{align}
	[u]_{\frac{1}{21}-\eps,\frac{1}{10}-\eps} < +\infty. \label{ineq:u}
	\end{align}
\end{Theorem}
\begin{Remark}
	Very recently, Buckmaster and Vicol used the technique of convex integration to prove non-uniqueness of weak solutions to the Navier-Stokes equation on $\mathbb{T}^3$ in \cite{BV17}. However, their solutions are not continuous, in contrast to the continuous weak solutions obtained here for \eqref{eq:Prandtl} and \eqref{eq:Euler-partial}. In fact, by Serrin's regularity criterion, any bounded weak solutions to 3D Navier-Stokes equation  must be regular, yielding the uniqueness.
\end{Remark}

We now make some comments on the analysis in this paper.

The main idea of the constructions here is to employ the convex integration technique introduced in  \cite{dLSz5} and subsequently developed in  \cite{IsettVicol,IsettOh16,Buckmaster2013transporting,Isett12,Isett16,Buckmaster2014} for the incompressible Euler system.
In the breakthrough work \cite{dLSz5},  De Lellis and Sz{\'e}kelyhidi employed high frequency Beltrami waves  as the principle building blocks  to construct continuous  Euler flows with non-conserved energy. Since then, the convex integration technique has been refined and applied to other systems of fluid \cite{IsettVicol,Isett12,Buckmaster2013transporting,Buckmaster2014,DaneriSzekelyhidi16,TZ17}. For a  thorough discussion, see \cite{dLSz16}. Isett \cite{Isett16} proved the Onsager's conjecture for the $3$-D Euler equations, constructing $C^{1/3-\eps}$ H\"{o}lder continuous Euler flows with non-conserved energy. A shorter proof was given by \cite{BDSV17}. Mikado waves, first introduced in \cite{DaneriSzekelyhidi16}, were employed as the main building blocks in \cite{Isett16,DaneriSzekelyhidi16}  replacing Beltrami waves in the previous constructions, along with a novel gluing approximation technique.

However,  the schemes for the Euler equations \cite{IsettVicol,Isett12,IsettOh16,Buckmaster2013transporting,dLSz5,Buckmaster2014,Isett16,DaneriSzekelyhidi16} may breakdown in the presence of viscosities. 
The main difficulties in developing a convex integration iteration scheme for the Prandtl system \eqref{eq:Prandtl} are vertical viscosities and that the pressure being fixed by boundary data instead of a Lagrange multiplier for the incompressible Euler or Navier-Stokes equations.
It seems that the schemes using Beltrami waves  in \cite{Buckmaster2013transporting,Buckmaster2014,dLSz4,dLSz5,Isett12} and the scheme using  Mikado waves in \cite{Isett16} are not directly applicable  to the system \eqref{eq:Prandtl}.

To deal with the transport-vertical-diffusion effect, our main observation is that a convex integration scheme could work using only  horizontal oscillations for the $3D$ Prandtl \eqref{eq:Prandtl}.
 We employ a serial convex integration scheme inspired by \cite{IsettVicol}, using localized linear plane waves as the main building blocks. The tangential and vertical length scales  of the flow are chosen to be compatible with the degenerate parabolic structure of the system \eqref{eq:Prandtl}. The serial nature of the scheme and the coupling of the convection  and the vertical diffusion restrict the regularity  obtained in our results.

In the very recent breakthrough \cite{BV17}, Buckmaster and Vicol obtained non-uniqueness of weak solutions to the three-dimensional Navier-Stokes equations. They developed a new convex integration scheme in Sobolev spaces using intermittent Beltrami flows which combined concentrations and oscillations. Later, the idea of using intermittent flows was used to study non-uniquenss for transport equations in \cite{MS17}, which used scaled Mikado waves. In view of these development, it seems natural to investigate the Prandtl system \eqref{eq:Prandtl} using  the technique of intermittent flows. However, it seems to us that the building blocks in \cite{BV17, MS17} do not directly work for the Prandtl system \eqref{eq:Prandtl} due to the difference in the pressure and the structure of the equations.

The rest of this paper is organized as follows.  In Section \ref{sec:iterationLemma}, the iteration lemma for constructing weak solutions to the Prandtl system  is stated. In Section \ref{sec:corrections}, we give the main constructions for the iteration lemma. In Section \ref{sec:estimates}, we prove the main estimates. In  Section \ref{sec:iterThm}, the main results are proved using the iteration lemma.

\subsection*{Notations}
The following notations are used in the rest of the paper.
Set
\begin{align*}
\R_+=(0,\infty), \quad \bRp=[0,\infty).
\end{align*}

 Let $\T^d$ denote $d$-dimensional torus with the volume normalized to unity:
\begin{align*}
|\mathbb{T}^d| = 1.
\end{align*}
Fix a set  of unit vectors in $\mathbb{R}^2$:
\begin{align}
\{\f_i\}_{i=1}^3 =\{(1,0),(0,1),\frac{1}{\sqrt{2}}(1,1)\} \subset \mathbb{R}^2. \label{def:f_i}
\end{align}
Then $\{ (1,0), (0,1) \}$ form a basis for $\mathbb{R}^2$ and   $\{\f_i\otimes \f_i\}_{i=1}^3$  form a basis for the space of symmetric $2\times2$ matrices, respectively.

For a set $\Omega \subset \Rp \times \T^2 \times \Rp$, let $P_x \Omega$ and $P_{t,y} \Omega$ denote its projection into $\T^2$ and $\Rp \times \Rp$ respectively.  Denote $\text{supp}_{t,y} g = P_{t,y} \supp g$ for a given function $g$. The tangential and the spatial gradient are denoted by
\begin{align*}
\nabla_x = (\frac{\ptl}{\ptl x_1}, \frac{\ptl}{\ptl x_2}), \quad \nabla = (\nabla_x, \ptl_y) = (\frac{\ptl}{\ptl x_1}, \frac{\ptl}{\ptl x_2},\frac{\ptl}{\ptl y}).
\end{align*}

\section{Brief Outline and The Main Iteration Lemma}\label{sec:iterationLemma}

\subsection{Brief Outline of the Scheme}

Adapting the convex integration method developed in  \cite{dLSz4,dLSz5,IsettVicol}, we will obtain a weak solution $(u,v)$ to  \eqref{eq:Prandtl} as the limit of  solutions $\{(u_{(n)},v_{(n)},S_{(n)},Y_{(n)})\}$ to the following approximate system,
\begin{align}\label{eq:PrStress}
\begin{cases}
\ptl_t u_{(n)} + \nabla_x \cdot (u_{(n)} \otimes u_{(n)} ) + \ptl_y(v_{(n)} u_{(n)}) - \ptl_{yy}^2 u_{(n)} + \nabla_x P&= \nabla_x \cdot S_{(n)}\\&+ \ptl_y Y_{(n)},\\
\nabla_x \cdot u_{(n)} + \ptl_y v_{(n)} = 0,
\end{cases}
\end{align}
where $S_{(n)}$ is a symmetric $2 \times 2$ matrix and $Y_{(n)}$ is a vector in $\mathbb
{R}^2$. The errors of the approximations are measured by the stress term $R_{(n)}=(S_{(n)},Y_{(n)})$.

In each step of the iteration, writing the stress in components as
\begin{align*}
S_{(n)}=-\sum_{i=1}^3 S_{(n),i}\f_{i}\otimes \f_{i}, \quad Y = -\sum_{i=1}^3 Y_{(n),i} \f_{i},
\end{align*}
we introduce high frequency waves in the forms
\begin{align*}
w_{(n+1)} = (u_{(n+1)}-u_{(n)},v_{(n+1)}-v_{(n)}) = \sum_I e^{i \lambda_{(n+1)} \xi_{(n+1),I}} \tilde{W}_{(n+1),I}
\end{align*}
to eliminate the largest components (in $L^{\infty}$ norms) of the stress $(S_{(n),i},Y_{(n),i})$ (which is taken to be $(S_{(n),1},Y_{(n),1})$ by renumbering the $i$-index). The new stress takes the form
\begin{align*}
(S_{(n+1)},Y_{(n+1)}) = -\Big(\sum_{i=2}^3S_{(n),i}\f_{i}\otimes \f_{i},\sum_{i=2}^3 Y_{(n),i}\f_{i}\Big) + \delta R_{(n+1)},
\end{align*}
where $\delta R_{(n+1)}$ is a small correction of the order $O(1/\lambda_{(n+1)})$, obtained by solving the divergence equations with oscillatory sources of frequency $O(\lambda_{(n+1)})$, .
Repeating this procedure and choosing the frequency parameters $\lambda_{n} \to \infty$, we can ensure that   the errors converge to zero uniformly, i.e., $\|R_{(n)}\|_{C^0} \to 0$. The precise outcome of a single iteration is  stated in the main iteration lemma below.

\subsection{The Main Iteration Lemma}

The frequency-energy levels for the approximate solution $(u,v,S,Y)$, adapted from \cite{Isett12,IsettVicol},  will be used in the iteration.

In the following, set
\begin{align}
\| \cdot \| = \| \cdot \|_{L^{\infty}}.
\end{align}

\begin{Definition}\label{Def:FE-Levels}
	Let $\Xi,\mcE_u,\mcE_1,\mcE_2,\mcE_3$ be positive numbers satisfying
	\begin{equation}
	\Xi \geq 1,\quad  4\mcE_3 \leq 2\mcE_2 \leq \mcE_1 \leq \mcE_u \label{E_i}.
	\end{equation}

	A smooth solution $(u,v,S,Y)$ to the system \eqref{eq:PrStress} 	is said to have frequency-energy levels below $(\Xi,\mcE)=(\Xi,\mcE_u,\mcE_1,\mcE_2,\mcE_3)$, if the following estimates are satisfied:
	\begin{equation}
	\begin{aligned}
	\|\nabla_{x} (u,v)\| \leq \Xi \mcE_u^{1/2},\quad \|(\ptl_t + u \cdot \nabla_{x} + v \ptl_y) (u,v)\| \leq \Xi \mcE_u, \\
	\| \ptl_y (u,v)\| \leq \Xi^{1/2} \mcE_u^{3/4},\quad \| \ptl_{yy}^2 (u,v)\| \leq \Xi \mcE_u, \label{est-u}
	\end{aligned}
	\end{equation}
	 and
	for $i=1,2,3$, any multi-indices $0 \leq |\alpha| + \beta + \gamma \leq 1$,
	\begin{equation}
	\begin{aligned}
	\|(\ptl_t + u \cdot \nabla_{x} + v \ptl_y)^{\gamma} \nabla_{x}^{\alpha}\ptl_y^{\beta} (S_i,Y_i)\| \leq \Xi^{|\alpha|} (\Xi^{1/2}\mcE_u^{1/4})^{\beta+2\gamma} \mcE_i, \label{est-r_i}
	\end{aligned}
	\end{equation}
	where   the stress $R=(S,Y)$ is written in components as
	\begin{align}
	S=-\sum_{i=1}^3 S_{i}\f_{i}\otimes \f_{i}, \quad Y = -\sum_{i=1}^3 Y_{i} \f_{i}. \label{eq:mbR}
	\end{align}
	Here the vector $\{\f_i\}_{i=1}^3$ are defined in \eqref{def:f_i} (possibly renumbering) and $Y_j = 0$ corresponding to  $\f_j = \frac{1}{\sqrt{2}}(1,1)$.
\end{Definition}

Now the main iteration lemma can be stated as
\begin{Lemma}\label{Lemma:Iteration}
	Given a positive constants $\vartheta$, there exists a constant $C_\vartheta$ depending on $\vartheta$  such that the following holds:\\
	Suppose that $(u,v,R)=(u,v,S,Y)$ 	is a smooth solution  to the system \eqref{eq:PrStress}  with frequency-energy levels below $(\Xi,\mcE)=(\Xi,\mcE_u,\mcE_1,\mcE_2,\mcE_3)$ and
	\begin{equation}
	\mcE_u \leq \Xi^2, \quad \ell := \Xi^{-1/2}\mcE_u^{-1/4} \leq (1+\|v\|_{L^{\infty}} )^{-1}. \label{eq:ell}
	\end{equation}
	Let $\e(t,y)$  be a given non-negative function satisfying
	\begin{equation}
	\begin{aligned}
	\e(t,y) &\geq 4 \mcE_1 \text{ on } N(\text{supp}_{t,y} R;\ell^2,\ell),
	\\
	N(\supp \e;50\ell^2,50\ell) &\subset \Rp \times \Rp,
	 \label{supp-e}
	\end{aligned}
	\end{equation}
	and for $0 \leq \alpha + \beta \leq 1$,
	\begin{align}
	\| (\ptl_t + u \cdot \nabla_{x} + v \ptl_y)^{\alpha} \ptl_y^{\beta} (\e^{1/2})\|_{L^{\infty}} \leq C_{\alpha,\beta}  \ell^{-(2\alpha + \beta)}\mcE_1^{1/2}.
	\label{est-e}
	\end{align}
	Then for any positive number $N$   such that
	\begin{equation}
	N \geq  \max\left\{\Xi^{\vartheta},  \left(\frac{\mcE_u }{\mcE_3}\right)^{3/2} \left(\frac{\mcE_1 }{\mcE_3}\right)^{3/2}\right\}, \label{ineq:N}
	\end{equation}
	there exists a smooth solution $(\tilde{u},\tilde{v},\tilde{R})$ to the system \eqref{eq:PrStress} with frequency-energy levels below  $(\tilde{\Xi},\tilde{\mcE})=(\tilde{\Xi},\tilde{\mcE_u},\tilde{\mcE_1},\tilde{\mcE_2},\tilde{\mcE_3})$, with
	\begin{equation}
	(\tilde{\Xi},\tilde{\mcE_u},\tilde{\mcE_1},\tilde{\mcE_2},\tilde{\mcE_3})=(C_{\vartheta} N \Xi,\mcE_1,2\mcE_2,2\mcE_3,N^{-1/3}\mcE_u^{1/2}\mcE_1^{1/2}). \label{tilde-FE-levels}
	\end{equation}
	Furthermore, the correction $w = (\tilde{u},\tilde{v}) - (u,v)$ satisfies the estimates
	\begin{equation}
	\begin{aligned}
	&\|\nabla_x^{\alpha} \ptl_y^{\beta} w\|_{L^{\infty}} \leq C_{\vartheta} (N \Xi)^{\alpha}(N^{1/2} \Xi^{1/2} \mcE_1^{1/4})^{\beta} \mcE_1^{1/2}, \quad 0 \leq |\alpha|+ \frac{\beta}{2} \leq 1, \\
	&\|(\ptl_t + \tilde{u}  \cdot \nabla_x + \tilde{v} \ptl_y ) w\|_{L^{\infty}} \leq C_{\vartheta} N \Xi \mcE_1, \label{est-w}
	\end{aligned}
	\end{equation}
	and the support of the constructions satisfies
	\begin{align}
	\text{supp}_{t,y} (w,\tilde{R}) \subset  N(\supp \e;\ell^2,\ell). \label{supp-w-tildeR}
	\end{align}
\end{Lemma}

\section{The Corrections}\label{sec:corrections}

\subsection{Preliminaries}
Given two positive numbers $a$ and  $b$, denote
\begin{align*}
a \lesssim b \iff a \leq C b,
\end{align*}
for some positive constant $C$ that is independent of the parameter $N$ and the frequency-energy levels $(\Xi,\mcE_u,\mcE_1,\mcE_2,\mcE_3)$ in Lemma \ref{Lemma:Iteration}.

Set the frequency parameter to be
\begin{equation}
\lambda = B^3 N \Xi,
\end{equation}
where $B>3$ is a constant to be chosen later.
The time, tangential and vertical length scale parameters of the constructions are chosen to be
\begin{equation}
\tau = B^{-1}\Xi^{-1}N^{-1/3}\mcE_u^{-1/2}, \ell_x = B^{-1}\Xi^{-1}N^{-1/3},  \ell_y = B^{-1} \Xi^{-1/2}N^{-1/3}\mcE_u^{-1/4}. \label{def:ell}
\end{equation}
It follows from \eqref{E_i} and  \eqref{ineq:N} that
\begin{align}
\tau  & = \ell \ell_y \leq  \ell_y = B^{-1}N^{-1/3} \ell \leq \ell \leq 1, \label{ineq:tau-ell}\\
N &\geq \left(\frac{\mcE_u }{\mcE_3}\right)^{3/2} \left(\frac{\mcE_1 }{\mcE_3}\right)^{3/2} \geq \left(\frac{\mcE_u}{\mcE_1}\right)^{3/2}. \label{ineq:N-Cor}
\end{align}

\subsubsection{Transport estimates} \label{subsubsec:transport-est} For later applications, some elementary transport estimates are recorded in this section, which are just anisotropic versions of those in \cite{Isett12}. The proofs are given in Appendix \ref{sec:transport-estimates}, for completeness.

\begin{Lemma}\label{Lemma:transport-scaling}
	Let $\ell_1,\cdots,\ell_d, \delta_U, \delta_f$ be positive numbers and $m \geq 1$ be a positive integer. Suppose that $\bar{U}(t,z)$ is a smooth vector field on $\R \times \R^d$, and $f(t,z)$ is a smooth solution to the Cauchy problem
	\begin{equation}
	(\ptl_t + \bar{U} \cdot \nabla_z)f = 0, \quad f|_{t=0}=f_0, \label{eq:transport-f}
	\end{equation}
	with the following estimates
	\begin{equation}
	\begin{aligned}
	\|\ptl_{z}^{\alpha} \bar{U}\|_{L^\infty}& \leq C_{\alpha} \ell^{-\alpha} \delta_{U}, \quad \text{ for } 1 \leq |\alpha| \leq m,  \\
	\|\ptl_{z}^{\alpha} f_0\|_{L^\infty}   & \leq C_{\alpha}\ell^{-\alpha}\delta_f ,\quad  \text{ for } 0 \leq |\alpha| \leq m,   \label{ineq:U-f_0}
	\end{aligned}
	\end{equation}
	where for any multi-index $\alpha=(\alpha_1,\cdots,\alpha_d)$,
	\[  \ptl_{z}^{\alpha} =  (\frac{\ptl}{\ptl z_1})^{\alpha_1} \cdots (\frac{\ptl}{\ptl z_d})^{\alpha_d}, \quad  \ell^{\alpha}=(\ell_1)^{\alpha_1}\cdots (\ell_d)^{\alpha_d}. \]
	Then, for $0 \leq |\alpha| \leq m$, it holds that
	\begin{align}
	\|\ptl_{z}^{\alpha} f(t,\cdot)\|_{L^\infty}   \leq \tilde{C}_{\alpha}\ell^{-\alpha}\delta_f , \quad \text{ for } |t|  \leq \delta_{U}^{-1} \min_i \ell_i, \label{ineq:ptl_z-f}
	\end{align}
	where $\tilde{C}_{\alpha}$ are functions of the constants $\{C_{\beta}: |\beta| \leq |\alpha|\}$.
\end{Lemma}

\begin{Lemma}\label{Lemma:flow-scaling}
	Let $\ell_1,\cdots,\ell_d, \delta_U$ be positive numbers. Suppose that $\bar{U}(t,z)$ is a smooth vector field on $\R \times \R^d$,  $\Phi_s(t,z) = (t+s,\Phi_s^1(t,z), \cdots, \Phi_s^d(t,z))$ is the flow generated by $(\ptl_t + \bar{U} \cdot\nabla_z )$, i.e., $\Phi_s$ is
	the unique solution to
	\begin{align}
	\frac{d}{ds} \Phi_s(t,z) = (1,\bar{U}(\Phi_s(t,z))), \quad \Phi_0(t,z)=(t,z), \quad \forall (t,z) \in \R \times \R^d, \label{def:Phi_s}
	\end{align}
	and $p=(p_1,\cdots,p_d), q=(q_1,\cdots,q_d) \in \mathbb{R}^d$ are two points such that
	\begin{align}
	\|\frac{\ptl}{\ptl z_i} \bar{U}\|_{L^\infty}\leq A_1 \ell_i^{-1} \delta_{U},\quad  |p_{i} - q_i|   \leq A_2 \ell_i, \quad \text{ for } i=1,\cdots,d, \label{ineq:U-p-q}
	\end{align}
	where $A_1, A_2$ are two positive constants.
	Then, for $i=1,\cdots,d$,
	\begin{align}
	|\Phi_s^i(t,p) - \Phi_s^i(t,q)|   \leq A_2 e^{A_1} \ell_i, \quad \text{ for } |s|  \leq \delta_{U}^{-1} \min_i \ell_i. \label{ineq:Phi_s}
	\end{align}
\end{Lemma}

\subsubsection{Mollifications and partitions of unity}
As in  \cite{ContidLSz12,dLSz5},   mollifications are employed  to deal with the potential loss of derivatives in the iteration. Set
\begin{align}
\mfD = P_{t,y}^{-1}(N(\supp \e; 2\ell^2, 2\ell)), \quad \mfD' = P_{t,y}^{-1}(N(\supp \e; 10\ell^2, 10\ell)). \label{def:mfD}
\end{align}
Let $(u^{ext},v^{ext})$ be a Lipschitz extension of $(u,v)$ in $\Rp \times \mathbb{T}^2 \times \mathbb{R}$ that coincides with $(u,v)$ in $\mfD'$.
Let $\bar{\eta} \in C_c^\infty(\mathbb{R})$ be a smooth even function such that
\begin{align*}
\supp \bar{\eta} \in (-3/4,3/4), \quad \bar{\eta}(s)=1, \text{ for } |s| < 5/8, \quad  \int \bar{\eta}(s)ds = 1,
\end{align*}
and set $\bar{\eta}_{\eps}(s) = \frac{1}{\eps}\bar{\eta}(\frac{s}{\eps})$ for any $\eps > 0$.
$(u,v)$ are mollified in space as follows:
\begin{align}
(u_\ell, v_\ell)(t,x,y) = ((u^{ext},v^{ext}) * \tilde{\eta}_{\ell})(t,x,y),
\label{def:u_ell-v_ell}\\
\text{where} \quad
\tilde{\eta}_{\ell}(x_1,x_2,y) = \bar{\eta}_{\ell_x}(x_1)\bar{\eta}_{\ell_x}(x_2)\bar{\eta}_{\ell_y}(y). \nonumber
\end{align}
Note that $(u_\ell, v_\ell)$ will be used only  in the domain $\mfD$, thus the choice of the extension makes no difference in the constructions. Clearly
\begin{align}
\nabla_x \cdot u_{\ell} + \ptl_y v_{\ell} = (\nabla_x \cdot u + \ptl_y v) * \tilde{\eta}_{\ell} = 0 \quad \text{ in } \mfD. \label{eq:div-free-u-v-ell}
\end{align}

It follows from the definitions \eqref{def:u_ell-v_ell} and the estimates \eqref{est-u} that
\begin{align}
\|(u - u_{\ell},v-v_{\ell})\|_{L^{\infty}(\mfD')} & \leq \ell_x \|\nabla_x (u,v)\|_{L^{\infty}} + \ell_y \|\ptl_y (u,v)\|_{L^{\infty}} \nonumber \\
& \leq B^{-1}N^{-1/3}\mcE_u^{1/2}, \label{ineq:u-u_ell}
\end{align}
and
\begin{equation}
\begin{aligned}
\|\nabla_x^{\alpha} \ptl_y^{\beta}\nabla_x(u_\ell, v_\ell)\|_{L^{\infty}(\mfD')}   & \leq \|\nabla_x^{\alpha} \ptl_y^{\beta} \tilde{\eta}\|_{L^{\infty}} \| \|\nabla_x (u,v)\|_{L^{\infty}}  \leq C_{\alpha,\beta} \ell_x^{-|\alpha|} \ell_y^{-\beta} \Xi \mcE_u^{1/2}, \\
\|\nabla_x^{\alpha} \ptl_y^{\beta}\ptl_y (u_\ell, v_\ell)\|_{L^{\infty}(\mfD')} &  \leq \|\nabla_x^{\alpha} \ptl_y^{\beta} \tilde{\eta}\|_{L^{\infty}} \| \|\ptl_y (u,v)\|_{L^{\infty}}  \leq C_{\alpha,\beta} \ell_x^{-|\alpha|} \ell_y^{-\beta} \Xi^{1/2} \mcE_u^{3/4}. \label{est-D_x-u_ell}
\end{aligned}
\end{equation}

The quadratic partitions of the unity adapted to the coarse flow $(u_\ell, v_\ell)$ are constructed explicitly below following \cite{Isett12}.
Let
\[\eta(s) = \frac{\bar{\eta}(s)}{\sqrt{\sum_{k \in \mathbb{Z}}\bar{\eta}^2(s-k)}}. \]
Then  $\{\eta(s-k):k \in \mathbb{Z}\}$ forms a quadratic partition of the unity satisfying
\[ \sum_{k \in \mathbb{Z}}\eta^2(s-k) = 1, \quad s \in \mathbb{R}. \]
For $\kappa_0 \in \mathbb{Z}$, set
\begin{align}
\eta_{\kappa_0}(t)=\eta\left(\tau^{-1}(t - \kappa_0\tau)\right). \label{def:eta_kappa_0}
\end{align}
Then $\eta_{\kappa_0}$ is supported in $((\kappa_0-\frac{3}{4}) \tau,(\kappa_0+\frac{3}{4}) \tau)$ with the estimates
\begin{align}
\|\frac{d^{\alpha}}{d t^{\alpha}} \eta_{\kappa_0}\|_{L^{\infty}} \leq C_{\alpha} \tau^{-\alpha}, \label{est-eta_k_0}
\end{align}
and $\{\eta_{\kappa_0}(t): \kappa_0 \in \mathbb{Z} \}$ forms a quadratic partition of the unity such that
\begin{align}
\sum_{\kappa_0 \in \mathbb{Z}}\eta_{\kappa_0}^2(t) \equiv 1, \quad t \in \mathbb{R}. \label{eq:eta_kappa_0}
\end{align}
For the tangential direction in the periodic setting, let $\mathbb{Z} \ni H \geq 0$ satisfy
\[ 2^{-H}  \leq \ell_x < 2^{-(H-1)}, \]
and let $\eta_H$ be the $2^H$-periodic extension of $\eta$, i.e., $\eta_H(s) = \sum_{k \in \mathbb{Z}}\eta(s-2^{H} k)$.
Denote the group $\mathbb{Z}$ modulo $2^{H}$ by $\mathbb{Z}/2^{H}\mathbb{Z}$. It is easy to see that
\[ \sum_{k \in \mathbb{Z}/2^{H}\mathbb{Z}}\eta_H^2(s-k) = 1. \]
For $\tilde{\kappa} = (\kappa_1,\kappa_2, \kappa_3) \in  (\mathbb{Z}/2^{H}\mathbb{Z})^2\times \mathbb{Z}$, set
\begin{align*}
\tilde{\psi}_{\tilde{\kappa}}(x_1,x_2,y) = \eta_H\left(2^H(x_1-\kappa_1 2^{-H})\right)\eta_H\left(2^H(x_2-\kappa_2 2^{-H})\right)\eta\left(\ell_y^{-1}(y-\kappa_3 \ell_y)\right),\\
\tilde{Q}_{\tilde{\kappa}}:=[(\kappa_1-\frac{3}{4}) \frac{1}{2^H},(\kappa_1+\frac{3}{4}) \frac{1}{2^H}] \times [(\kappa_2-\frac{3}{4}) \frac{1}{2^H},(\kappa_2+\frac{3}{4}) \frac{1}{2^H}] \times [(\kappa_3-\frac{3}{4}) \ell_y,(\kappa_3+\frac{3}{4}) \ell_y].
\end{align*}
It is easy to check that $\{\tilde{\psi}_{\tilde{\kappa}} \in C_c^\infty(\tilde{Q}^{\tilde{\kappa}})\}$ forms a quadratic partition of the unity adapted to the cubes $\{\tilde{Q}_{\tilde{\kappa}}:\tilde{\kappa}   \in  (\mathbb{Z}/2^{H}\mathbb{Z})^2\times \mathbb{Z}\}$, such that
\begin{align*}
\sum_{\tilde{\kappa}   \in  (\mathbb{Z}/2^{H}\mathbb{Z})^2\times \mathbb{Z}}\tilde{\psi}_{\tilde{\kappa}}^2(x,y) \equiv 1, \text{ for } (x,y) \in \mathbb{T}^2 \times \mathbb{R}, \quad \text{ and } \quad \supp \tilde{\psi}_{\tilde{\kappa}} \subset \tilde{Q}_{\tilde{\kappa}}.
\end{align*}
Furthermore, the following estimates hold for any $\alpha, \beta \geq 0$:
\begin{equation}
\| \nabla_{x}^\alpha \ptl_y^\beta \tilde{\psi}_{\tilde{\kappa}} \|_{L^{\infty}} \leq C(\alpha,\beta)\ell_x^{-\alpha} \ell_y^{-\beta}. \label{est-psi-o}
\end{equation}

Let $\Phi_s$ be the flow generated by the mollified space-time vector field $(\ptl_t + u_\ell \cdot\nabla_x + v_{\ell} \ptl_y)$, i.e. $\Phi_s$ is the unique solution to
\begin{align}
\frac{d}{ds} \Phi_s(t,x,y) = (1,u_\ell(\Phi_s(t,x,y)),v_\ell(\Phi_s(t,x,y))), \quad \Phi_0(t,x,y)=(t,x,y). \label{def:barPhi_s}
\end{align}

For $\kappa = (\kappa_0,\kappa_1,\kappa_2, \kappa_3) \in \mathbb{Z} \times (\mathbb{Z}/2^{H}\mathbb{Z})^2\times \mathbb{Z}$, let $Q_{\kappa}$ be the image of $\tilde{Q}_{\tilde{\kappa}}$ under the coarse flow map $\Phi$, and $q_{\kappa}$ its `center', i.e.,
\begin{equation}
\begin{aligned}
Q_{\kappa} &= \{\Phi_{s}(\kappa_0 \tau,x,y): |s| < \frac{3}{4}\tau, (x,y) \in \tilde{Q}_{\tilde{\kappa}} \}, \\
q_{\kappa} &= (t_{\kappa},x_{\kappa},y_{\kappa}) := (\kappa_0 \tau,\kappa_1 2^{-H},\kappa_2 2^{-H}, \kappa_3 \ell_y) \in Q_{\kappa} \label{def:p_kappa}.
\end{aligned}
\end{equation}
Define $\psi_{\kappa}$ to be the unique solution to
\begin{align}
(\ptl_t + u_{\ell} \cdot \nabla_x + v_{\ell} \ptl_y ) \psi_{\kappa} = 0,  \quad \psi_{\kappa}(\kappa_0\tau,x,y) = \tilde{\psi}_{\tilde{\kappa}}(x,y). \label{eq:psi_kappa}
\end{align}
Then
\[ \frac{d}{dt}\psi_{\kappa}(\Phi_{t-\kappa_0 \tau}(\kappa_0 \tau,x,y)) = 0.   \]
Recalling that $\supp \tilde{\psi}_{\tilde{\kappa}} \subset \tilde{Q}^{\tilde{\kappa}}$, one has
\begin{gather}
\supp \eta_{\kappa_0}\psi_{\kappa} \subset Q_{\kappa} \label{supp-eta-psi_k}, \\
\sum_{(\kappa_1,\kappa_2,\kappa_3) \in (\mathbb{Z}/2^{H}\mathbb{Z})^2\times \mathbb{Z}}\psi_\kappa^2(t,x,y) \equiv 1.
\end{gather}
It follows from \eqref{eq:eta_kappa_0} that
\begin{align}
\sum_{\kappa}\eta_{\kappa_0}^2(t)\psi_{\kappa}^2(t,x,y) = \sum_{\kappa_0}\eta_{\kappa_0}(t)^2\sum_{\kappa_1, \kappa_2, \kappa_3}\psi_{\kappa}^2(t,x,y) \equiv 1. \label{sum_eta-psi}
\end{align}
Denote the mollified vector field by
\begin{align}
\barDt &= \ptl_t + u_\ell(t,x,y) \cdot \nabla_x + v_{\ell}(t,x,y) \ptl_y.
\end{align}
It follows from \eqref{est-u}, \eqref{eq:ell} and \eqref{ineq:u-u_ell} that
\begin{align}
  \| \barDt (u,v) \|_{L^{\infty}(\mfD')} \nonumber &\leq  \| (\ptl_t + u \cdot \nabla_x + v \ptl_y ) (u,v) \| \\
& + \|u - u_{\ell}\|_{L^{\infty}(\mfD')} \|\nabla_x (u,v)\| + \|v-v_{\ell}\|_{L^{\infty}(\mfD')} \|\ptl_y (u,v)\| \nonumber \\
\leq & \Xi \mcE_u + B^{-1}N^{-1/3}\Xi \mcE_u + B^{-1}N^{-1/3}\Xi^{1/2} \mcE_u^{5/4} \lesssim \Xi \mcE_u. \label{est-u-Dt}
\end{align}
As a consequence of \eqref{est-u},  \eqref{est-D_x-u_ell}, and \eqref{est-u-Dt}, one has, for $(t,x,y) \in \mfD$,
\begin{align}
	&|\barDt \nabla_x^{\alpha} \ptl_y^{\beta}(u_\ell, v_\ell)(t,x,y)| \nonumber \\
	&= | \iint  \nabla_x^{\alpha} \ptl_y^{\beta} \tilde{\eta}_{\ell}(x',y') \Big\{ \barDt (u,v)(t,x - x',y - y') \nonumber \\
	 &\quad +  (u_{\ell}(t,x,y) - u_{\ell}(t,x-x',y-y')) \cdot \nabla_x  (u,v)(t,x - x',y - y') \nonumber \\
	 &\quad +  (v_{\ell}(t,x,y) - v_{\ell}(t,x-x',y-y')) \ptl_y  (u,v)(t,x - x',y - y')   \Big\}dx'dy' | \nonumber \\
	&\leq \iint \nabla_x^{\alpha} \ptl_y^{\beta} \tilde{\eta}_{\ell}(x',y') \Big\{ \| \barDt (u,v) \|_{L^{\infty}(\mfD')}  \nonumber \\
	& \quad \quad + (|x'| \|\nabla_{x} u_{\ell}\|_{L^{\infty}(\mfD')} + |y'| \| \ptl_y u_{\ell} \|_{L^{\infty}(\mfD')})\|\nabla_x  (u,v)\|_{L^{\infty}(\mfD')}\nonumber \\
	& \quad \quad   + (|x'| \|\nabla_{x} v_{\ell}\|_{L^{\infty}(\mfD')} + |y'| \| \ptl_y v_{\ell} \|_{L^{\infty}(\mfD')})\| \ptl_y (u,v) \|_{L^{\infty}(\mfD')} \Big\}  dx'dy' \nonumber \\
	&\leq C(\alpha,\beta) \ell_x^{-\alpha} \ell_y^{-\beta} \Big\{ \| \barDt (u,v) \| + (\ell_x \|\nabla_{x} u_{\ell}\| + \ell_y \| \ptl_y u_{\ell} \|)\| \nabla_x  (u,v)\|\nonumber \\
	& \quad \quad  + (\ell_x \|\nabla_{x} v_{\ell} \| + \ell_y \| \ptl_y v_{\ell} \|)\| \ptl_y  (u,v)\|  \Big\} \nonumber \\
	& \leq C(\alpha,\beta) \ell_x^{-\alpha} \ell_y^{-\beta} \Xi \mcE_u \label{est-barDtDxDy}
\end{align}
where \eqref{eq:ell} has been used in the last inequality.
In view of the following commuting relations
\begin{align}
[\nabla_x, \barDt] = \nabla_x u_{\ell} \cdot \nabla_x + (\nabla_x v_{\ell}) \ptl_y, \quad [\ptl_y, \barDt] = \ptl_y u_{\ell} \cdot \nabla_x + (\ptl_y v_{\ell}) \ptl_y, \label{eq:comm-D_x-D_t}
\end{align}
and the estimates \eqref{est-D_x-u_ell}, one can write
\begin{align}
\nabla_x^{\alpha} \ptl_y^{\beta}\barDt&= \barDt\nabla_x^{\alpha} \ptl_y^{\beta} + \sum_{\mathclap{a+b=\alpha, |b| \geq 1}} \nabla_x^{a} [\nabla_x, \barDt] \nabla_x^{b-1} \ptl_y^{\beta}   + \sum_{\mathclap{a+b=\beta, |b| \geq 1}} \nabla_x^{\alpha} \ptl_y^{a}[\ptl_y, \barDt] \ptl_y^{b-1}, \nonumber\\
=& \barDt\nabla_x^{\alpha} \ptl_y^{\beta} + \sum_{\mathclap{a+b=\alpha, |b| \geq 1}} C_{a,b}(\nabla_x^{a+1} u_{\ell} \cdot \nabla_x^{b} \ptl_y^{\beta} + \nabla_x^{a+1} v_{\ell} \cdot \nabla_x^{b-1} \ptl_y^{\beta+1}) \nonumber \\
\quad & + \sum_{\mathclap{c+d = \alpha, a+b=\beta, |b| \geq 1}} C_{a,b,c,d}(\nabla_x^{c} \ptl_y^{a+1} u_{\ell} \cdot \nabla_x^{d+1} \ptl_y^{b-1} + \nabla_x^{c} \ptl_y^{a+1} v_{\ell} \cdot \nabla_x^{d} \ptl_y^{b}). \label{eq:barDtDxDy}
\end{align}
It follows from \eqref{eq:ell},  \eqref{est-D_x-u_ell} and \eqref{est-barDtDxDy} that, for $|\alpha|,|\alpha'|,\beta,\beta' \geq 0$,
\begin{align}
&\|\nabla_x^{\alpha} \ptl_y^{\beta}\barDt\nabla_x^{\alpha'} \ptl_y^{\beta'}(u_\ell, v_\ell)\|_{L^{\infty}(\mfD)} \nonumber  \lesssim \|\barDt\nabla_x^{\alpha + \alpha'} \ptl_y^{\beta+\beta'}(u_\ell, v_\ell)\|\\
& \quad  +   \sum_{\mathclap{a+b=\alpha, |b| \geq 1}} \|(\nabla_x^{a+1} u_{\ell} \cdot \nabla_x^{\alpha'+b} \ptl_y^{\beta+\beta'} + \nabla_x^{a+1} v_{\ell} \cdot \nabla_x^{\alpha'+b-1} \ptl_y^{\beta+\beta'+1})(u_\ell, v_\ell)\| \nonumber \\
& \quad + \sum_{\mathclap{c+d = \alpha, a+b=\beta, |b| \geq 1}} \|(\nabla_x^{c} \ptl_y^{a+1} u_{\ell} \cdot \nabla_x^{d+\alpha'+1} \ptl_y^{\beta'+b-1} + \nabla_x^{c} \ptl_y^{a+1} v_{\ell} \cdot \nabla_x^{d+\alpha'} \ptl_y^{\beta'+b})(u_\ell, v_\ell)\| \nonumber \\
& \leq C(\alpha,\beta,\alpha',\beta') \ell_x^{-|\alpha+\alpha'|} \ell_y^{-(\beta+\beta')} \Xi \mcE_u. \label{est-D_t-u_ell}
\end{align}

Recalling \eqref{def:ell}, \eqref{est-psi-o} and  \eqref{eq:psi_kappa},
applying Lemma \ref{Lemma:transport-scaling} to $\psi_{\kappa}$ with
\begin{align*}
\bar{U}=(u_{\ell},v_{\ell}),\quad t=t-\kappa_0\tau,\quad  z_1=x_1,\quad z_2=x_2, \quad z_3=y,\\
\ell^1 = \ell^2 = \ell_x,\quad \ell^3 = \ell_y, \quad \delta_U = \mcE_U^{1/2}, \quad \delta_f = 1,
\end{align*}
one has the following transport estimates, for any $Q_{\kappa} \subset \mfD'$,
\begin{equation}
\|\nabla_{x}^\alpha \ptl_y^\beta \psi_\kappa \|_{L^{\infty}(Q_{\kappa})}\leq C(\alpha,\beta)\ell_x^{-\alpha} \ell_y^{-\beta}, \text{ for any } \alpha, \beta \geq 0. \label{est-psi}
\end{equation}
It follows from  the identities
\eqref{eq:barDtDxDy} and $\barDt \psi_\kappa \equiv 0$ that
\begin{align}
&\barDt \nabla_x^{\alpha} \ptl_y^{\beta}\psi_\kappa  = - \sum_{\mathclap{a+b=\alpha, |b| \geq 1}} \nabla_x^{a} [\nabla_x, \barDt] \nabla_x^{b-1} \ptl_y^{\beta} \psi_\kappa   - \sum_{\mathclap{a+b=\beta, |b| \geq 1}} \nabla_x^{\alpha} \ptl_y^{a}[\ptl_y, \barDt]  \ptl_y^{b-1} \psi_\kappa \nonumber \\
 = & -\sum_{\mathclap{a+b=\alpha, |b| \geq 1}} C_{a,b}(\nabla_x^{a+1} u_{\ell} \cdot \nabla_x^{b} \ptl_y^{\beta} + \nabla_x^{a+1} v_{\ell} \cdot \nabla_x^{b-1} \ptl_y^{\beta+1})\psi_\kappa \nonumber \\
 \quad & - \sum_{\mathclap{c+d = \alpha, a+b=\beta, |b| \geq 1}} C_{a,b,c,d}(\nabla_x^{c} \ptl_y^{a+1} u_{\ell} \cdot \nabla_x^{d+1} \ptl_y^{b-1} + \nabla_x^{c} \ptl_y^{a+1} v_{\ell} \cdot \nabla_x^{d} \ptl_y^{b})\psi_\kappa. \label{Dt-psi'}
\end{align}
Due to \eqref{eq:ell}, \eqref{est-D_x-u_ell} and \eqref{est-psi}, it holds that, for any $Q_{\kappa} \subset \mfD'$,
\begin{align}
\|\barDt \nabla_{x}^\alpha \ptl_y^\beta \psi_\kappa \|_{L^{\infty}(Q_{\kappa})}\leq C(\alpha,\beta)\ell_x^{-\alpha} \ell_y^{-\beta} \tau^{-1}. \label{Dt-psi-1}
\end{align}
Using the identities \eqref{eq:barDtDxDy} and \eqref{Dt-psi'}, one can write
\begin{align*}
&\barDt \nabla_x^{\alpha} \ptl_y^{\beta} \barDt \nabla_x^{\alpha'} \ptl_y^{\beta'} \psi_\kappa \\
=& \barDt \Big\{\barDt\nabla_x^{\alpha + \alpha'} \ptl_y^{\beta+\beta'} + \sum_{\mathclap{a+b=\alpha, |b| \geq 1}} C_{a,b}(\nabla_x^{a+1} u_{\ell} \cdot \nabla_x^{\alpha'+b} \ptl_y^{\beta+\beta'} + \nabla_x^{a+1} v_{\ell} \cdot \nabla_x^{\alpha'+b-1} \ptl_y^{\beta+\beta'+1}) \nonumber \\
\quad & + \sum_{\mathclap{c+d = \alpha, a+b=\beta, |b| \geq 1}} C_{a,b,c,d}(\nabla_x^{c} \ptl_y^{a+1} u_{\ell} \cdot \nabla_x^{d+\alpha'+1} \ptl_y^{\beta'+b-1} + \nabla_x^{c} \ptl_y^{a+1} v_{\ell} \cdot \nabla_x^{d+\alpha'} \ptl_y^{\beta'+b})\Big\}\psi_\kappa\\
= & \barDt \Big\{ \sum_{\mathclap{a+b=\alpha+\alpha', |b| \geq 1}} C_{a,b}(\nabla_x^{a+1} u_{\ell} \cdot \nabla_x^{b} \ptl_y^{\beta+\beta'} + \nabla_x^{a+1} v_{\ell} \cdot \nabla_x^{b-1} \ptl_y^{\beta+\beta'+1}) \nonumber \\
\quad & + \sum_{\mathclap{c+d = \alpha+\alpha', a+b=\beta+\beta', |b| \geq 1}} C_{a,b,c,d}(\nabla_x^{c} \ptl_y^{a+1} u_{\ell} \cdot \nabla_x^{d+1} \ptl_y^{b-1} + \nabla_x^{c} \ptl_y^{a+1} v_{\ell} \cdot \nabla_x^{d} \ptl_y^{b}) \Big\}\psi_\kappa.
\end{align*}
It follows from the estimates \eqref{eq:ell}, \eqref{est-D_x-u_ell}, \eqref{est-D_t-u_ell}, \eqref{est-psi} and \eqref{Dt-psi-1} that for $0 \leq \gamma+\gamma' \leq 2$, $|\alpha|,|\alpha'|,\beta,\beta' \geq 0$, and any $Q_{\kappa} \subset \mfD$,
\begin{align}
\|(\barDt)^{\gamma}\nabla_x^{\alpha} \ptl_y^{\beta} (\barDt)^{\gamma'}\nabla_x^{\alpha'} \ptl_y^{\beta'}\psi_\kappa \|_{L^{\infty}(Q_{\kappa})}
\leq C(\alpha,\beta,\alpha',\beta')
\ell_x^{-|\alpha+\alpha'|}\ell_y^{-(\beta+\beta')}\tau^{-(\gamma+\gamma')}. \label{Dt-psi}
\end{align}

\subsection{Corrections of the velocity}
The new velocity is chosen to be
\begin{align*}
(\tilde{u},\tilde{v})= (u,v) + (U,V) = (u,v) + w.
\end{align*} Here the correction $w=(U,V)$ is the sum of individual waves $w_I$ of the form
\begin{equation}
w = \sum_I w_I = \sum_I e^{ i \lambda \xi_I}\tilde{W}_I = \sum_I e^{ i \lambda \xi_I} (\tilde{U}_I,\tilde{V}_I), \label{eq:w}
\end{equation}
where  $w_I$ are divergence-free localized plane waves supported in $Q_{\kappa(I)}$ with phase functions $\xi_I$.  The index for $w_I$ takes the form
\begin{align*}
I=(\kappa(I),s(I)) \in  \mathbb{Z} \times (\mathbb{Z}/2^{H}\mathbb{Z})^2\times \mathbb{Z} \times \{ +,- \},
\end{align*}
where the index $\kappa(I)=(\kappa_0(I),\kappa_1(I),\kappa_2(I), \kappa_3(I)) \in \mathbb{Z} \times (\mathbb{Z}/2^{H}\mathbb{Z})^2\times \mathbb{Z}$ indicates the space-time location of $w_I$ and $s(I) \in \{ +,- \}$ specifies its oscillating direction.
The profile $\tilde{W}_I$ takes the form
\begin{equation}
\tilde{W}_I =  W_I +
\delta W_I, \label{eq:tilde-W}
\end{equation}
where $W_I$ is the main part and $\delta W_I$ is a small correction to ensure the divergence-free condition.
Set
\begin{gather}
U_{\kappa(I)} = \eta_I \psi_I a_I \f_1 = \eta_{\kappa_0(I)} \psi_{\kappa(I)} a_{\kappa(I)} \f_1, ~
V_{\kappa(I)} = \eta_I \psi_I b_I  = \eta_{\kappa_0(I)} \psi_{\kappa(I)} b_{\kappa(I)}, \label{def:U-V}\\
W_{I} = W_{\kappa(I)} = (U_{\kappa(I)},V_{\kappa(I)}) = \eta_I(t) \psi_I(t,x,y) A_I, ~ A_I = A_{\kappa(I)} = (a_{\kappa(I)}\f_1,b_{\kappa(I)}). \nonumber
\end{gather}
Here $\f_1$ is the unit $2$-vector in \eqref{eq:mbR}, and  $\eta_{\kappa_0}, \psi_\kappa$ are the partitions of unity in \eqref{def:eta_kappa_0} and \eqref{eq:psi_kappa}, respectively. The amplitude functions $a_{\kappa(I)}$ and $b_{\kappa(I)}$ are defined to be
\begin{align}
a_{\kappa(I)}  = \sqrt{\frac{(\e+S_1)(q_{\kappa(I)})}{2}}, \quad
b_{\kappa(I)}  = \frac{Y_1(q_{\kappa(I)})}{2a_{\kappa(I)}}, \label{def:a}
\end{align}
where $q_{\kappa}$ is the center of $Q_\kappa$ defined in \eqref{def:p_kappa}. Note that $Y_1(q_{\kappa(I)}) = 0$ if $q_{\kappa(I)} \not \in \supp R$.
It follows from \eqref{est-r_i}, \eqref{supp-e},  and \eqref{est-e} that  $a_I,b_I$ are well-defined  with
\begin{align*}
a_{I} & \leq  (\|\e\|_{L^{\infty}}^{1/2} + \|S_1\|_{L^{\infty}}^{1/2})/\sqrt{2} \leq C  \mcE_1^{1/2},\\
b_{I} & \leq  \|Y_1\|_{L^{\infty}} ( \inf_{\supp R} \e^{1/2})^{-1} \leq \mcE_1\mcE_1^{-1/2} =    \mcE_1^{1/2}.
\end{align*}
Thus we obtain the estimates for $A_I = (a_{\kappa(I)}\f_1,b_{\kappa(I)})$:
\begin{align}
|A_I| \leq C  \mcE_1^{1/2}.   \label{est-omega}
\end{align}
The phase function $\xi_I$ is a linear function defined as
\begin{align}
\xi_I = \xi_I(t,x) = s(I) [\kappa(I)]\f_1^\perp \cdot (x - u_{\kappa(I)} t), \label{def:xi_I}
\end{align}
where
\begin{align}
(u_I,v_I) = (u_{\kappa(I)},v_{\kappa(I)})=(u_\ell,v_\ell)(q_{\kappa}),\\
[\kappa]= [\kappa_0,\kappa_1,\kappa_2,\kappa_3]=\sum_{j=0}^3 2^{j}[\kappa_j] + 1, \quad
[\kappa_j] = \begin{cases}
0 & \text{ if } \kappa_j \text{ is even},\\
1 & \text{ if } \kappa_j \text{ is odd}.
\end{cases} \label{def:[kappa]}
\end{align}
The definition of $[k]$ ensures that $e^{i \lambda \xi_{I}}$ and $e^{i \lambda \xi_{J}}$ are separated in frequencies whenever $Q_{\kappa(I)} \cap Q_{\kappa(J)} \neq \emptyset$ and $J \neq \bar{I}$, where the conjugate index is defined by
\begin{align*}
\bar{I}=(k(I),-s(I)), \quad \forall I = (k(I),s(I)).
\end{align*}
For any such indices $I$ and $J$, one can verify  that
\begin{align}
1 \leq |\nabla_x (\xi_I + \xi_J)|  \leq [\kappa(I)]+[\kappa(J)] \leq 32. \label{ineq:xi_I-xi_J}
\end{align}
Notice that $\xi_I$ solves the following transport equation with constant coefficients:
\begin{align}
(\ptl_t + u_I \cdot \nabla_x + v_I \ptl_y )\xi_I = 0. \label{eq:xi_I}
\end{align}
Furthermore, the following orthogonality condition holds:
\[ \f_1 \cdot \nabla_x \xi_I = 0, \]
which ensures that $e^{ i \lambda \xi_I}W_I$ is divergence-free to the leading order of $\lambda$.

To find the small corrections $\delta W_I$,
we define
\begin{equation}
\begin{aligned}
w_I' &= -\Delta ( \frac{1}{\lambda^2|\nabla \xi_I|^2}e^{i\lambda\xi_I}W_I) = - \frac{1}{\lambda^2|\nabla \xi_I|^2}\Delta ( e^{i\lambda\xi_I}W_I), \\
w_I'' &= \nabla \big(\nabla \cdot (\frac{1}{\lambda^2|\nabla \xi_I|^2}e^{i\lambda\xi_I}W_I)\big) = \frac{1}{\lambda^2|\nabla \xi_I|^ 2} \nabla \big(\nabla \cdot (e^{i\lambda\xi_I}W_I)\big), \label{def:w_I}
\end{aligned}
\end{equation}
and set
\begin{align*}
w_I = w_I' + w_I''.
\end{align*}
It follows from  the definitions and \eqref{supp-eta-psi_k} that $\overline{w_I} = w_{\overline{I}}$, with
\begin{align}
\nabla \cdot w_I = 0, \quad \supp w_I \subset \supp W_I \subset  Q_{\kappa(I)}. \label{eq:w_I-div-free}
\end{align}
Therefore, the correction $w = \sum_I w_I$
is real-valued and divergence-free.
Furthermore, the following expressions hold for $w_I$:
\begin{align}
w_I &=  e^{i\lambda \xi_I}(W_I +
\delta W_I) = e^{i\lambda \xi_I}( \eta_I \psi_I A_I + \eta_I A_I  \sum_{1 \leq  |\beta| \leq 2}  C_{I,\beta}\lambda^{-|\beta|}  \ptl_{x,y}^{\beta}\psi_I ), \label{eq:w_I}
\end{align}
where $C_{I,\beta}$ are constants given by
\begin{align}
C_{I,\beta} &= \sum_{\alpha:|\alpha+\beta|=2} |\nabla_x \xi_I |^{-2} C_{\alpha,\beta}i^{|\alpha|}(\ptl_x \xi_I)^{\alpha}. \nonumber
\end{align}
Indeed, direct computations give
\begin{align*}
w_I' &= e^{i\lambda\xi_I}W_I+\frac{1}{\lambda^2|\nabla \xi_I |^2}\sum_{|\alpha+\beta|=2, |\beta| \geq 1}C_{\alpha,\beta}(\ptl_{x,y}^{\alpha}e^{i\lambda\xi_I})(\ptl_{x,y}^{\beta}W_I) \nonumber\\
&=e^{i\lambda\xi_I}W_I + |\nabla \xi_I |^{-2}\sum_{|\alpha+\beta|=2, |\beta| \geq 1} C_{\alpha,\beta}\lambda^{-2}(\ptl_{x,y}^{\alpha}e^{i\lambda\xi_I})(\ptl_{x,y}^{\beta}W_I),
\end{align*}
where
\begin{align*}
\ptl_{x,y}^{\beta} W_{I} = \eta_I(t) A_I  \ptl_{x,y}^{\beta}\psi_I(t,x,y) .
\end{align*}
Since $\xi=\xi_I(t,x)$ is linear in $t$ and $x$, so
\begin{align*}
\ptl_y (e^{i\lambda \xi_I}) = 0, \quad
\ptl_{x}^{\alpha}(e^{i\lambda \xi_I}) = i^{|\alpha|}\lambda^{|\alpha|}(\nabla_x \xi_I)^\alpha e^{i\lambda \xi_I},
\end{align*}
where $(\nabla_x \xi_I)^{\alpha}=(\ptl_{x_1} \xi_I)^{\alpha_1}(\ptl_{x_2} \xi_I)^{\alpha_2}$ for $\alpha=(\alpha_1,\alpha_2)$.
Thus
\begin{align*}
w_I' = e^{i\lambda\xi_I}\big(W_I+ |\nabla \xi_I |^{-2} \eta_I A_I  \sum_{\mathclap{|\alpha+\beta|=2, |\beta| \geq 1}}  C_{\alpha,\beta}i^{|\alpha|}(\ptl_x \xi_I)^{\alpha}\lambda^{-|\beta|}  \ptl_{x,y}^{\beta}\psi_I \big).
\end{align*}
Recalling that $\f_1\cdot \nabla_x \xi_\kappa = 0$, one gets
\begin{align*}
w_I'' & = \frac{1}{\lambda^2|\nabla \xi_I|^2}\nabla \big(\nabla \cdot (e^{i\lambda\xi_I}W_I)\big)=\frac{1}{\lambda^2|\nabla \xi_I|^2}\nabla  (e^{i\lambda\xi_I}\nabla \cdot W_I) \nonumber\\
&=e^{i\lambda \xi_I}|\nabla_x \xi_I |^{-2} \eta_I A_I  \sum_{\mathclap{|\alpha+\beta|=2, |\beta| \geq 1}} C_{\alpha,\beta}i^{|\alpha|}(\ptl_x \xi_I)^{\alpha}\lambda^{-|\beta|} \ptl_{x,y}^{\beta}\psi_I.
\end{align*}
Consequently, \eqref{eq:w_I} follows.

\subsection{The equations for the new stress}
To obtain the equation for $\tilde{R}$,  we set $R = (S,Y), (\tilde{u},\tilde{v})= (u,v) + (U,V)$ and use that $(u,v,S,Y)$ solves \eqref{eq:PrStress} to obtain
\begin{align*}
\nabla \cdot \tilde{R} =&  \ptl_t \tilde{u} + \nabla_x \cdot (\tilde{u} \otimes \tilde{u} ) + \ptl_y(\tilde{v} \tilde{u}) - \ptl_{yy}^2 \tilde{u} + \nabla_x P\\
=& \nabla_x \cdot (S + U \otimes U ) + \ptl_y (Y + V U ) + (\ptl_t U + \nabla_x \cdot (u \otimes U) + \ptl_y(v U)) \\
&+ \nabla_x \cdot (U \otimes u) + \ptl_y(V u) - \ptl_{yy}^2 U.
\end{align*}
From the expression $w = \sum_I e^{ i \lambda \xi_I} (\tilde{U}_I,\tilde{V}_I)$, one can separate the interactions of waves into the low and high frequency parts as
\begin{align*}
U \otimes U &=\sum_{I,J} e^{i \lambda (\xi_I+\xi_J)} \tilde{U}_I \otimes \tilde{U}_J  = \sum_I  \tilde{U}_I \otimes \tilde{U}_{\bar{I}} + \sum_{J \neq \bar{I}} e^{i \lambda (\xi_I+\xi_J)}\tilde{U}_I \otimes \tilde{U}_J,\\
V U &=  \sum_{I,J} e^{i \lambda (\xi_I+\xi_J)} \tilde{V}_I  \tilde{U}_J = \sum_I \tilde{V}_I \tilde{U}_{\bar{I}} + \sum_{J \neq \bar{I}}e^{i \lambda (\xi_I+\xi_J)} \tilde{V}_I  \tilde{U}_J,
\end{align*}
where $\bar{I}=(k(I),-s(I))$.
Hence,
\begin{align*}
\nabla \cdot \tilde{R}=& \nabla_x \cdot (S + \sum_I  \tilde{U}_I \otimes \tilde{U}_{\bar{I}}) + \ptl_y (Y + \sum_I \tilde{V}_I \tilde{U}_{\bar{I}}) \\
& + \sum_{J \neq \bar{I}}\nabla_x \cdot( e^{i \lambda (\xi_I+\xi_J)}\tilde{U}_I \otimes \tilde{U}_J) + \sum_{J \neq \bar{I}}\ptl_y (e^{i \lambda (\xi_I+\xi_J)} \tilde{V}_I  \tilde{U}_J) \\
& + (\ptl_t U + \nabla_x \cdot (u \otimes U) + \ptl_y(v U)) + \nabla_x \cdot (U \otimes u) + \ptl_y(V u) - \ptl_{yy}^2 U.
\end{align*}
Due to \eqref{eq:mbR}, the new stress $\tilde{R}$  can be decomposed as two parts:
\begin{align}
\tilde{R} &=  -\Big(\sum_{i\neq 1}S_{i}\f_{i}\otimes \f_{i},\sum_{i\neq 1}Y_{i}\f_{i}\Big) + \delta R, \label{def:tlR}\\
 &= -\Big(\sum_{i\neq 1}S_{i}\f_{i}\otimes \f_{i},\sum_{i\neq 1}Y_{i}\f_{i}\Big) + (R_S + R_{M}  + R_H + R_T + R_{L}).
\end{align}
Let $\e_{\ell}$ and $(S_{\ell},Y_{\ell})$ be the mollifications of $\e$ and $(S_1,Y_1)$ defined as
\begin{equation}
\begin{aligned}
e_{\ell}(t,x,y)&= \sum_{\kappa} \eta_{\kappa}^2(t)\psi_{\kappa}^2(t,x,y)e(q_{\kappa}),\\
(S_{\ell},Y_{\ell})(t,x,y)&=\sum_{\kappa} \eta_{\kappa}^2(t)\psi_{\kappa}^2(t,x,y) (S_1,Y_1)(q_{\kappa}). \label{def:R_ell}
\end{aligned}
\end{equation}
$\delta R = R_S + R_{M}  + R_H + R_T + R_{L}$ is required to solve the following divergence equations:
\begin{align}
\nabla \cdot R_S =&   \nabla \cdot \Big(-(\e_{\ell}+S_{\ell}) \f_1 \otimes \f_1 + \sum_I  \tilde{U}_I \otimes \tilde{U}_{\bar{I}},   -Y_{\ell} \f_1 + \sum_I \tilde{V}_I \tilde{U}_{\bar{I}}\Big), \label{eq:R_S}  \\
\nabla \cdot R_{M}=& \nabla \cdot  \Big( (\e_{\ell} - \e + S_\ell-S_1) \f_1 \otimes \f_1, (Y_\ell-Y_1) \f_1\Big) \nonumber \\
&  \quad + \nabla \cdot \Big( (u-u_{\ell}) \otimes U + U \otimes (u-u_{\ell}),V (u-u_{\ell})+(v-v_{\ell})U \Big), \label{eq:R_M}  \\
\nabla \cdot R_H =  & \sum_{J \neq \bar{I}}\nabla_x \cdot(e^{i \lambda (\xi_I+\xi_J)} \tilde{U}_I \otimes \tilde{U}_J) + \sum_{J \neq \bar{I}}  \ptl_y (e^{i \lambda (\xi_I+\xi_J)}  \tilde{V}_I  \tilde{U}_J), \label{eq:R_H}   \\
\nabla \cdot R_T=& \ptl_t U + \nabla_x \cdot (u_{\ell} \otimes U) + \ptl_y(v_{\ell} U) - \ptl_{yy}^2 U, \label{eq:R_T} \\
\nabla \cdot R_{L}= & \nabla_x \cdot ( U \otimes u_{\ell}) + \ptl_y (V u_{\ell}). \label{eq:R_L}
\end{align}
Here, as in \cite{IsettVicol}, the term $\e_{\ell} f_1 \otimes \f_1$ is added to ensure  the coefficient $\e_{\ell}+S_{\ell} \geq 0$.

\section{Estimates of the new velocity and stress}\label{sec:estimates}
In this section, we estimate the new velocity and stress in order to prove Lemma \ref{Lemma:Iteration}. When there is no need to retain the dependence on $B$ explicitly, as in the estimates for the derivatives of $w,\tilde{u},\tilde{v}$ and $R$, we write $C_B$ for some generic polynomial functions of $B$.  The constant $C_{\vartheta}$  in Lemma \ref{Lemma:Iteration} depending on $\vartheta$ and $C_B$  will be determined at the end of this section.
\subsection{Estimates of the new velocity} \label{subsec:est-velocity}
Now we estimate the supports of the corrections.
Recall that $\Phi_s(t,x,y) = (s+t,\Phi_s^1(t,x,y),\Phi_s^2(t,x,y),\Phi_s^3(t,x,y))$ is the flow generated by $(\ptl_t + u_\ell \cdot\nabla_x + v_{\ell} \ptl_y)$ defined in \eqref{def:barPhi_s}.
Setting
\begin{equation}
\begin{aligned}
\bar{U}=(u_{\ell},v_{\ell}), \quad (z_1,z_2,z_3)=(x_1,x_2,y), \\
\ell_1 = \ell_2 = \ell_x, \ell_3 = \ell_y, \quad  \delta_U = \mcE_U^{1/2},\quad A_1 = B^{-1}N^{-1/3}, A_2 = \frac{3}{2}, \label{set:flow-scaling}
\end{aligned}
\end{equation}
it follows from Lemma \ref{Lemma:flow-scaling} and  \eqref{est-D_x-u_ell}  that for $ |s| \leq \tau, (x,y) \in \tilde{Q}_{\tilde{\kappa}(I)}$,
\begin{align}
|\Phi_s^3(t_I,x,y) - y_I| &\leq   |\Phi_s^3(t_I,x,y) - \Phi_s^3(t_I,x_I,y_I)|  + |\Phi_s^3(t_I,x_I,y_I) - \Phi_0^3(t_I,x_I,y_I)| \nonumber \\
&\leq A_2 e^{A_1} \ell_y + s \max|\frac{d}{ds}\Phi_s^3| \leq 3\ell_y + \tau \|v\|_{L^{\infty}} \leq 4 \ell_y, \label{def:A}
\end{align}
where we have used \eqref{eq:ell}, \eqref{ineq:tau-ell} for the last inequality and denoted $(t_I,x_I,y_I)=q_{\kappa(I)}$ for $q_{\kappa(I)}$  in  \eqref{def:p_kappa}.
It follows from the definition of $Q_{\kappa}$ in \eqref{def:p_kappa} that
\begin{align}
P_{t,y}(Q_{\kappa}) \subset N(\{(t_{\kappa},y_{\kappa})\};\tau, 4 \ell_y). \label{supp-Q_kappa}
\end{align}
Recalling from \eqref{supp-eta-psi_k} that $\supp \eta_{\kappa_0}\psi_{\kappa} \subset Q_{\kappa}$, so one can get
\begin{align}
\text{supp}_{t,y} (\eta_{\kappa_0}^2(\cdot) \psi_{\kappa}^2(\cdot) (S_1,Y_1)(q_{\kappa})) \subset   N(\text{supp}_{t,y} R;\tau, 4\ell_y). \label{supp-S_I}
\end{align}
It follows from  the definitions \eqref{eq:ell}, \eqref{def:ell} and \eqref{def:R_ell} that
\begin{align}
\text{supp}_{t,y} (S_{\ell},Y_{\ell}) \subset N(\text{supp}_{t,y} R;\tau, 4\ell_y) \subset N(\text{supp}_{t,y} R;\frac{1}{2}\ell^2,\frac{1}{2}\ell), \label{supp-S_ell-Y_ell}
\end{align}
if $B$ is chosen  so that  $B \geq 8$. Similarly, it holds that
\begin{align}
\text{supp}_{t,y} e_{\ell} \subset N(\supp \e;\tau,4\ell_y) \subset N(\supp \e;\frac{1}{2}\ell^2,\frac{1}{2}\ell).  \label{supp-e_ell}
\end{align}
It follows from \eqref{supp-e} and \eqref{def:U-V} that
\begin{equation}
\text{supp}_{t,y} w   \subset   \bigcup_{\kappa} \supp_{t,y}( \eta_{\kappa_0}^2(\cdot)\psi_{\kappa}^2(\cdot) (S_1+\e)(q_{\kappa})) \subset  N(\supp \e;\frac{1}{2}\ell^2,\frac{1}{2}\ell). \label{supp-w}
\end{equation}
Due to \eqref{def:mfD} and \eqref{supp-Q_kappa},  $P_{t,y}(Q_{\kappa}) \cap N(\supp \e;\frac{1}{2}\ell^2,\frac{1}{2}\ell) = \emptyset$, for any $Q_{\kappa} \nsubseteq \mfD$. Thus it follows from \eqref{supp-w} that
\begin{align}
w_I \equiv 0  \quad \text{ for any } Q_{\kappa(I)} \nsubseteq \mfD. \label{supp-w_I-Q_I}
\end{align}
It is easy to see from \eqref{eq:psi_kappa} that the number of non-zero $\eta_I\psi_I$  is at most $2^4=64$ at any point. Hence to estimate $w = \sum_I w_I$, it suffices to estimate each \ $w_I$ for any $Q_{\kappa(I)} \subset \mfD$.

Recalling the expression \eqref{eq:w_I},
it follows from \eqref{est-psi} and  \eqref{est-omega} that
\begin{align}
	\|W_I\|_{C^0} &\lesssim \mcE_1^{1/2}, \label{est-W_I}\\
	\|\delta W_I\|_{C^0} &\lesssim \sum_{1 \leq  \beta + \beta' \leq 2} \lambda^{-(\beta+\beta')} \ell_x^{-\beta} \ell_y^{-\beta'} \mcE_1^{1/2} \lesssim \sum_{1 \leq  |\beta| \leq 2} (\lambda\ell_x )^{-|\beta|} \mcE_1^{1/2} \nonumber \\
	&  \lesssim B^{-2}N^{-2/3}\mcE_1^{1/2}, \label{est-deltaW_I}
	\end{align}
where one has used the fact $\ell_x \leq \ell_y$ from \eqref{eq:ell}.
Thus the correction $w$ is bounded by
\begin{align}
\|w\|_{C^0} \lesssim \|W_I+\delta W_I\|_{C^0} \leq C \mcE_1^{1/2}. \label{est-w-C^0}
\end{align}
Due to \eqref{eq:w_I}, one has
\begin{align*}
\ptl_{x,y}^\alpha \tilde{W}_I &=  \eta_I  A_I \ptl_{x,y}^\alpha \psi_I + \eta_I A_I  \sum_{1 \leq  |\beta| \leq 2}  C_{I,\beta}\lambda^{-|\beta|}  \ptl_{x,y}^{\alpha+\beta}\psi_I.
\end{align*}
It follows from \eqref{est-psi} and \eqref{est-omega} that for $|\alpha|, \beta \geq 0$,
\begin{align}
\|\nabla_{x}^\alpha \ptl_y^\beta  \tilde{W}_I  \|_{C^0} &\leq C_{\alpha,\beta}  \ell_x^{-\alpha} \ell_y^{-\beta}\mcE_1^{1/2}, \label{est-tilde-W}\\
\|\nabla_{x}^\alpha \ptl_y^\beta \delta W_I  \|_{C^0} &\leq C_{\alpha,\beta}  \ell_x^{-\alpha} \ell_y^{-\beta}  B^{-2}N^{-2/3}\mcE_1^{1/2}.\label{est-deltaW_I-d}
\end{align}
For the derivatives of  $\tilde{W}_I$ involving $\barDt$,
it follows from the expression \eqref{eq:w_I} and the estimates \eqref{est-eta_k_0}, \eqref{Dt-psi} and \eqref{est-omega} that for $0 \leq \gamma+\gamma' \leq 2$,
\begin{align}
\|(\barDt)^{\gamma}\nabla_x^{\alpha} \ptl_y^{\beta} (\barDt)^{\gamma'}\nabla_x^{\alpha'} \ptl_y^{\beta'} \tilde{W}_I \|_{L^{\infty}(Q_I)} &\lesssim  \ell_x^{-|\alpha+\alpha'|}\ell_y^{-(\beta+\beta')}\tau^{-(\gamma+\gamma')}\mcE_1^{1/2}, \label{est-W_I-Dt} \\
\|(\barDt)^{\gamma}\nabla_x^{\alpha} \ptl_y^{\beta} (\barDt)^{\gamma'}\nabla_x^{\alpha'} \ptl_y^{\beta'} \delta \tilde{W}_I \|_{L^{\infty}(Q_I)} &\lesssim  \ell_x^{-|\alpha+\alpha'|}\ell_y^{-(\beta+\beta')}\tau^{-(\gamma+\gamma')}B^{-2}N^{-2/3}\mcE_1^{1/2}.\label{est-deltaW_I-Dt}
\end{align}

Now we estimate the derivatives of $w$ and $(\tilde{u},\tilde{v})$.
It follows from \eqref{est-tilde-W} that
\begin{align}
\|\nabla_x w\|_{C^0} &\lesssim \|i \lambda (\nabla_x \xi_I) e^{i\lambda \xi_I} \tilde{W}_I + e^{i\lambda \xi_I} \nabla_x \tilde{W}_I \|_{C^0} \lesssim \lambda \mcE_1^{1/2} \nonumber \\
&\leq C_B N \Xi \mcE_1^{1/2}. \label{ineq:ptl_x-w}
\end{align}
It thus follows from \eqref{est-u} and \eqref{ineq:N-Cor}  that
\begin{align}
\|\nabla_x (\tilde{u},\tilde{v})\|_{C^0} &\leq \|\nabla_x (u,v) \|_{C^0}  +  \|\nabla_x w\|_{C^0} \nonumber  \leq \Xi \mcE_u^{1/2} +  C_B N \Xi \mcE_1^{1/2} \\
&\leq C_B N \Xi \mcE_1^{1/2}.\label{ineq:ptl_x-tilde-u}
\end{align}
Similarly, \eqref{est-u}, \eqref{ineq:N-Cor} and \eqref{est-tilde-W} imply that
\begin{align}
\|\ptl_y w\|_{C^0} &\lesssim \|e^{i\lambda \xi_I} \ptl_y \tilde{W}_I \|_{C^0}  \lesssim \ell_y^{-1} \mcE_1^{1/2} \leq C_B N^{1/3}\Xi^{1/2}\mcE_u^{1/4}\mcE_1^{1/2} \nonumber \\
& \leq C_B N^{1/2}\Xi^{1/2}\mcE_1^{3/4}, \label{ineq:ptl_y-w}\\
\|\ptl_y (\tilde{u},\tilde{v})\|_{C^0} &\leq \|\ptl_y  (u,v)\|_{C^0}  +  \|\ptl_y  w\|_{C^0} \leq \Xi^{1/2} \mcE_u^{3/4} + C_B N^{1/2}\Xi^{1/2}\mcE_1^{3/4} \nonumber \\
& \leq C_B N^{1/2}\Xi^{1/2}\mcE_1^{3/4}. \label{ineq:ptl_y-tilde-u}
\end{align}
Similar estimates yield also
\begin{align*}
\|\ptl_{yy}^2 w\|_{C^0} &\lesssim \|e^{i\lambda \xi_I} \ptl_{yy}^2 \tilde{W}_I \|_{C^0}  \lesssim \ell_y^{-2} \mcE_1^{1/2} \leq C_B N^{2/3}\Xi\mcE_u^{1/2}\mcE_1^{1/2}  \leq C_B N \Xi \mcE_1,\\
\|\ptl_{yy}^2 (\tilde{u},\tilde{v}) \|_{C^0} &\leq \|\ptl_{yy}^2  (u,v)\|_{C^0}  +  \|\ptl_{yy}^2  w\|_{C^0} \leq \Xi \mcE_u + C_B N \Xi \mcE_1 \leq C_B N \Xi \mcE_1.
\end{align*}
To estimate the material derivatives, one will use the following lemma.
\begin{Lemma}\label{Lemma:mol-errors}
	For any $Q_{\kappa} \subset \mfD$, it holds that
	\begin{align}
	\|(u_{\ell},v_{\ell})(\cdot) - (u_{\ell},v_{\ell})(q_{\kappa})\|_{L^{\infty}(Q_{\kappa})} & \leq C B^{-1}N^{-1/3}\mcE_u^{1/2}, \label{ineq:u_ell-u_I}\\
	\|(e(\cdot)-e(q_{\kappa}),S_1(\cdot)-S_1(q_{\kappa}),Y_1(\cdot)-Y_1(q_{\kappa}))\|_{L^{\infty}(Q_{\kappa})} & \leq C B^{-1}N^{-1/3}\mcE_1. \label{ineq:r_1-r_ell}
	\end{align}
\end{Lemma}
\begin{proof}
	Denote $(t_{\kappa},x_{\kappa},y_{\kappa})=q_{\kappa}$ for $q_{\kappa}$ in  \eqref{def:p_kappa}. Recall that $\Phi_s$ is the flow generated by $(\ptl_t + u_\ell \cdot\nabla_x + v_{\ell} \ptl_y)$ in \eqref{def:barPhi_s}. It follows from \eqref{est-D_x-u_ell},  \eqref{est-D_t-u_ell} and the mean value theorem that,
	for $(t_{\kappa},x,y) \in \tilde{Q}_{\tilde{\kappa}}, |s| \leq \tau$,
	\begin{align*}
	&|(u_{\ell},v_{\ell})(\Phi_s(t_{\kappa},x,y)) - (u_{\ell},v_{\ell})(t_{\kappa},x_{\kappa},y_{\kappa})|\nonumber\\
	=&|(u_{\ell},v_{\ell})(\Phi_s(t_{\kappa},x,y)) - (u_{\ell},v_{\ell})(t_{\kappa},x,y)| +  |(u_{\ell},v_{\ell})(t_{\kappa},x,y)-(u_{\ell},v_{\ell})(t_{\kappa},x_{\kappa},y_{\kappa})| \\
	\leq&  \ell_x \|\nabla_{x} (u_{\ell},v_{\ell})\| + \ell_y \| \ptl_y (u_{\ell},v_{\ell}) \| +  \tau \| \barDt (u_{\ell},v_{\ell}) \| \leq C B^{-1}N^{-1/3}\mcE_u^{1/2}.
	\end{align*}
	Since $Q_{\kappa}=\{\Phi_s(t_{\kappa},x,y):(t_{\kappa},x,y) \in \tilde{Q}_{\tilde{\kappa}}, |s| \leq \tau\}$, so  \eqref{ineq:u_ell-u_I} follows.
	
	Since $\supp (e,S_1,Y_1) \subset \supp e \subset \mfD$ due to \eqref{supp-e}, it follows from the estimates \eqref{est-r_i}, \eqref{eq:ell}, \eqref{est-e},  and \eqref{ineq:u-u_ell}  that
	\begin{align}
	\| \barDt (e,S_1,Y_1)\|_{L^{\infty}} \leq & \| (\ptl_t + u \cdot \nabla_x + v \ptl_y) (e,S_1,Y_1) \| + \|u - u_{\ell}\|\|\nabla_x (e,S_1,Y_1)\|\nonumber \\
	&\quad + \|v - v_{\ell}\|\|\ptl_y (e,S_1,Y_1)\|\nonumber\\
	\leq & \Xi \mcE_u^{1/2} \mcE_1 + B^{-1}N^{-1/3}\mcE_u^{1/2} \mcE_1 (\Xi  + \Xi^{1/2} \mcE_u^{1/4}) \nonumber \\
	\leq & C \Xi \mcE_u^{1/2} \mcE_1. \label{est-r_1-Dt}
	\end{align}
	Using the same proof for \eqref{ineq:u_ell-u_I} with \eqref{est-r_i} and \eqref{est-e}, one obtains  \eqref{ineq:r_1-r_ell}.
\end{proof}

It follows from \eqref{eq:xi_I} and $\ptl_y \xi_I(x,t) = 0$ that
\[ \barDt \xi_I = (u_{\ell}-u_I) \cdot \nabla_x \xi_I.\]
Hence,
\begin{align}
&\barDt(e^{i\lambda \xi_I}\tilde{W}_I)
=i \lambda e^{i\lambda \xi_I}\tilde{W}_I  (u_{\ell}-u_I) \cdot \nabla_x \xi_I + e^{i\lambda \xi_I}\barDt \tilde{W}_I. \label{eq:Dt-W_I}
\end{align}
It follows from  \eqref{ineq:N-Cor} , \eqref{est-W_I-Dt} and \eqref{ineq:u_ell-u_I} that
\begin{align}
\|\barDt w\|_{C^0} &\lesssim \sup_{I: Q_{\kappa(I)} \subset \mfD} \|\barDt(e^{i\lambda \xi_I}\tilde{W}_I)\|_{L^{\infty}}  \lesssim \ \lambda \|\tilde{W}_I \|\|u_{\ell}-u_I\|_{L^{\infty}(Q_{\kappa})} + \|\barDt \tilde{W}_I\| \nonumber \\
& \leq C_B (N^{2/3}\Xi \mcE_U^{1/2} \mcE_1^{1/2} + N^{1/3}\Xi \mcE_U^{1/2} \mcE_1^{1/2}) \leq C_B N \Xi \mcE_1. \label{est-w-Dt}
\end{align}
Writing $\ptl_t + \tilde{u} \cdot \nabla_x + \tilde{v} \ptl_y = \barDt + (u-u_{\ell})\cdot \nabla_x + (v-v_{\ell})\ptl_y + w \cdot \nabla$, one obtains from  the  estimates \eqref{eq:ell}, \eqref{ineq:u-u_ell},  \eqref{est-w-C^0}, \eqref{ineq:ptl_x-w}, \eqref{ineq:ptl_y-w}, and \eqref{est-w-Dt} that
\begin{align*}
\|(\ptl_t + \tilde{u} \cdot \nabla_x + \tilde{v} \ptl_y ) w\| &\leq \|\barDt w\| + \|u-u_{\ell}\|\|\nabla_x w\| + \|v-v_{\ell}\|\|\ptl_y w\| + \|w\|\|\nabla w\| \\
&\leq C_B N \Xi \mcE_1.
\end{align*}
Since $\ptl_t + \tilde{u} \cdot \nabla_x + \tilde{v} \ptl_y  = (\ptl_t + u \cdot \nabla_x +v \ptl_y ) + w \cdot \nabla$, it follows from \eqref{est-u} and \eqref{est-w-C^0} that
\begin{align*}
\|(\ptl_t + \tilde{u} \cdot \nabla_x + \tilde{v} \ptl_y) (\tilde{u},\tilde{v})\| &\leq \|(\ptl_t + u \cdot \nabla_x + v \ptl_y )(u,v)\| + \|w\| \|\nabla (u,v)\| \\
&\quad + \|(\ptl_t + \tilde{u} \cdot \nabla_x + \tilde{v} \ptl_y) w\| \\
&\leq  C_B N \Xi \mcE_1.
\end{align*}
This completes the estimates for $w$ and $(\tilde{u},\tilde{v})$.

\subsection{Estimates of the errors}
Recall that the new stress $\tilde{R}$ is given by
\begin{align*}
\tilde{R} &=  -\Big(\sum_{i\neq 1}S_{i}\f_{i}\otimes \f_{i},\sum_{i\neq 1}Y_{i}\f_{i}\Big) + \delta R,
\end{align*}
and we decompose $\delta R = (R_S + R_{M}) +  ( R_H + R_T + R_{L})$
 into two parts. The first part  will be estimated below, and the second part will be handled by solving divergence equations with oscillatory sources in Section \ref{subsec:NewStressDivEq}.
\subsubsection{Mollification errors} \label{sec:moll-error}

Recall that
\begin{equation}
\begin{aligned}
R_M  &=  \Big( (e_{\ell} - e + S_\ell-S_1) \f_1 \otimes \f_1, (Y_\ell-Y_1) \f_1\Big)\\
& \quad  \quad + \Big((u-u_{\ell}) \otimes U + U \otimes (u-u_{\ell}),V (u-u_{\ell})+(v-v_{\ell})U\Big)  \\
&:= R_M' + R_M''. \label{def:mbR_M}
\end{aligned}
\end{equation}
The supports of $R_M$ can be estimate from \eqref{supp-S_ell-Y_ell}, \eqref{supp-e_ell} and \eqref{supp-w} as
\begin{align*}
\supp_{t,y} R_M &\subset \supp_{t,y} w \cup \supp \e_{\ell} \cup \supp \e \cup \text{supp}_{t,y} (S_1, Y_1)\cup \text{supp}_{t,y} (S_{\ell}, Y_{\ell}) \\\
&\subset  N(\supp \e;\frac{1}{2}\ell^2,\frac{1}{2}\ell).
\end{align*}
Then  \eqref{ineq:r_1-r_ell} implies that
\begin{align*}
\|R_M'\|_{C^0} &= \|(\sum_{\kappa}\eta_{\kappa}^2\psi_{\kappa}^2(e(q_{\kappa})-e(\cdot)+S_1(q_{\kappa})-S_1(\cdot)),\sum_{\kappa}\eta_{\kappa}^2\psi_{\kappa}^2(Y_1(q_{\kappa})-Y_1(\cdot))\|\\
&\lesssim \sup_{\kappa: Q_{\kappa} \subset \mfD} \|(e(\cdot)-e(q_{\kappa}),S_1(\cdot)-S_1(q_{\kappa}),Y_1(\cdot)-Y_1(q_{\kappa}))\|_{C^0(Q_{\kappa})} \\
&\lesssim B^{-1}N^{-1/3} \mcE_1.
\end{align*}
While \eqref{ineq:u-u_ell} and \eqref{est-w-C^0} lead to
\begin{align*}
\|R_M''\|_{C^0}
&\lesssim \|u-u_{\ell}\|_{L^{\infty}(\mfD)}\|U\| + \|V\| \|u-u_{\ell}\|_{L^{\infty}(\mfD)} + \|v-v_{\ell}\|_{L^{\infty}(\mfD)} \|U\|\\
&\lesssim B^{-1} N^{-1/3}\mcE_u^{1/2}\mcE_1^{1/2}.
\end{align*}
Therefore,
\begin{align}
\|R_M\|_{C^0}
& \leq C B^{-1} N^{-1/3}\mcE_u^{1/2}\mcE_1^{1/2}. \label{est-R_M-C^0}
\end{align}

Now we estimate the derivatives of $R_M$.
It follows from \eqref{est-r_i}, \eqref{est-e}, \eqref{est-psi} and \eqref{ineq:r_1-r_ell} that
\begin{align*}
\|\ptl_y R_M'\|
&\lesssim  \sup_{\kappa: Q_{\kappa} \subset \mfD} \|\ptl_y \psi_{\kappa}\|_{C^0(Q_{\kappa})} \|(e(\cdot)-e(q_{\kappa}),S_1(\cdot)-S_1(q_{\kappa}),Y_1(\cdot)-Y_1(q_{\kappa}))\|_{C^0(Q_{\kappa})} \\
&\quad + \|(\ptl_y e, \ptl_y S_1, \ptl_y Y_1) \| \\
& \leq  C_B \Xi^{1/2}\mcE_u^{1/4}\mcE_1 \leq C_B N^{1/6}\Xi^{1/2}\mcE_u^{1/2}\mcE_1^{3/4}.
\end{align*}
Due to \eqref{est-u},  \eqref{ineq:N-Cor}, \eqref{ineq:u-u_ell}, \eqref{est-D_x-u_ell}, \eqref{est-w-C^0} and \eqref{ineq:ptl_y-w}, one gets
\begin{align*}
\|\ptl_y R_M''\| &\lesssim \|(u-u_{\ell},v-v_{\ell})\|_{L^{\infty}(\mfD)} \|\ptl_y (U,V) \| + \|\ptl_y (u-u_{\ell},v-v_{\ell})\|_{L^{\infty}(\mfD)} \|(U,V)\|\\
&\leq C_B \Xi^{1/2}\mcE_u^{3/4}\mcE_1^{1/2} \leq C_B N^{1/6}\Xi^{1/2}\mcE_u^{1/2}\mcE_1^{3/4}.
\end{align*}
Hence,
\begin{align*}
\|\ptl_y R_M \|_{C^0} &\leq C_B N^{1/6}\Xi^{1/2}\mcE_u^{1/2}\mcE_1^{3/4} \leq C_B \tilde{\Xi}^{1/2} \tilde{\mcE}_u^{1/4}\tilde{\mcE}_3.
\end{align*}
In the same manner, it follows from \eqref{est-u}, \eqref{est-r_i}, \eqref{ineq:u-u_ell}, \eqref{est-D_x-u_ell}, \eqref{est-psi}, \eqref{est-w-C^0}, \eqref{ineq:ptl_x-w}, and \eqref{ineq:r_1-r_ell} that
\begin{align*}
\|\nabla_x R_M \|_{C^0} &\leq C_B N^{2/3}\Xi\mcE_u^{1/2}\mcE_1^{1/2} \leq C_B \tilde{\Xi} \tilde{\mcE}_3.
\end{align*}
Since $\barDt\psi_{\kappa} = 0$, it follows from \eqref{est-eta_k_0}, \eqref{ineq:r_1-r_ell} and \eqref{est-r_1-Dt} that
\begin{align*}
\|\barDt R_M'\| &\lesssim \sup_{\kappa} \|\ptl_t \eta_{\kappa} \| \|(e(\cdot)-e(q_{\kappa}),S_1(\cdot)-S_1(q_{\kappa}),Y_1(\cdot)-Y_1(q_{\kappa}))\|_{C^0(Q_{\kappa})} \\
&\quad + \|\barDt (e,S_1,Y_1) \| \lesssim  \Xi \mcE_u^{1/2} \mcE_1.
\end{align*}
As a consequence of \eqref{ineq:u-u_ell}, \eqref{est-u-Dt},\eqref{est-barDtDxDy}, \eqref{est-D_t-u_ell}, \eqref{est-w-C^0} and \eqref{est-w-Dt}, it holds that
\begin{align*}
\|\barDt R_M''\| &\lesssim \|(u-u_{\ell},v-v_{\ell})\|\|\barDt (U,V) \| + \|\barDt (u-u_{\ell},v-v_{\ell})\|\|(U,V)\| \\
& \leq C_B N^{2/3}\Xi \mcE_u^{1/2} \mcE_1.
\end{align*}
Thus,
\begin{align*}
\|\barDt R_M \|_{C^0} &\leq C_B N^{2/3}\Xi \mcE_u^{1/2} \mcE_1 \leq C_B \tilde{\Xi} \tilde{\mcE}_u^{1/2}\tilde{\mcE}_3.
\end{align*}

\subsubsection{The stress term} \label{sec:stress-term}

Note that
\begin{align*}
\sum_I (\tilde{U}_I \otimes \tilde{U}_{\bar{I}} + \tilde{V}_I \tilde{U}_{\bar{I}}) = 2\sum_{\kappa}\Big( (U_{\kappa}+\delta U_{\kappa})\otimes (U_{\kappa}+\delta U_{\kappa}) +  (V_{\kappa} + \delta V_{\kappa}) (U_{\kappa}+\delta U_{\kappa})\Big).
\end{align*}
It follows from the definitions of $R_S$ in \eqref{eq:R_S} that
\begin{align*}
R_S =& \Big(-(\e_{\ell}+S_{\ell}) \f_1 \otimes \f_1 + 2\sum_{\kappa}  U_{\kappa} \otimes U_{\kappa},  -Y_{\ell} \f_1 + 2\sum_{\kappa} V_{\kappa} U_{\kappa}\Big) +  R_{S'},\\
R_{S'}=&2\Big(\sum_{\kappa}\delta U_{\kappa} \otimes U_{\kappa} +
 U_{\kappa} \otimes \delta U_{\kappa} + \delta U_{\kappa} \otimes\delta U_{\kappa},0\Big)\\
&+2\Big(0,\sum_{\kappa} \delta V_{\kappa} U_{\kappa} + V_{\kappa} \delta U_{\kappa} + \delta V_{\kappa} \delta U_{\kappa}\Big).
\end{align*}
The constructions in \eqref{def:a}, \eqref{def:R_ell}, and \eqref{supp-w} yield
\begin{align*}
-(\e_{\ell}+S_{\ell})\f_1\otimes \f_1+ \sum_{\kappa} 2 U_{\kappa}\otimes U_{\kappa}
&= \sum_{\kappa}\eta_{\kappa_0}^2\psi_{\kappa}^2(-(\e+S_{1})(q_{\kappa})+2a_{\kappa}^2)(\f_1\otimes \f_1) = 0, \\
-Y_{\ell} \f_1 + 2\sum_{\kappa} V_{\kappa} U_{\kappa}&= \sum_{\kappa}\eta_{\kappa_0}^2\psi_{\kappa}^2(-Y_1(q_{\kappa})+2a_{\kappa}b_{\kappa})\f_1 = 0.
\end{align*}
Thus
\begin{equation}
R_S = R_{S'}. \label{eq:R_S'}
\end{equation}
Then \eqref{est-W_I} and \eqref{est-deltaW_I} imply  that
\begin{align}
\|R_S\|_{C^0}=\|R_S'\|_{C^0} &\lesssim \|W_I\|\|\delta W_I\| + \|\delta W_I\|\|\delta W_I\| \lesssim B^{-2}N^{-2/3}\mcE_1.
\end{align}
Now we estimate derivatives of $R_{S}$.
It follows from \eqref{ineq:N-Cor},  \eqref{est-tilde-W}, and \eqref{est-deltaW_I-d} that
\begin{align*}
\|\ptl_y R_S \|_{C^0} &= \|\ptl_y R_S' \|_{C^0} \lesssim \|\ptl_y W_I \| \|\delta W_I \| + \| W_I\| \|\ptl_y \delta W_I \| + \|\ptl_y \delta W_I \| \|\delta W_I \|\\
&  \lesssim N^{-1/3}\Xi^{1/2} \mcE_u^{1/4} \mcE_1  \leq C_B \tilde{\Xi}^{1/2} \tilde{\mcE}_u^{1/4}\tilde{\mcE}_3,\\
\|\nabla_x R_S \|_{C^0} &\lesssim \|\nabla_x W_I \| \|\delta W_I \| + \| W_I\| \|\nabla_x \delta W_I \| + \|\nabla_x \delta W_I \| \|\delta W_I \|\\
& \lesssim  N^{-1/3}\Xi \mcE_1 \leq C_B \tilde{\Xi} \tilde{\mcE}_3.
\end{align*}
Similarly, making use of \eqref{ineq:N-Cor}, \eqref{est-W_I-Dt}, and \eqref{est-deltaW_I-Dt}, one can get
\begin{align*}
\|\barDt R_S \|_{C^0} &\lesssim \|\barDt W_I \| \|\delta W_I \| + \| W_I\| \|\barDt \delta W_I \| + \|\barDt \delta W_I \| \|\delta W_I \|\\
& \lesssim \Xi N^{-1/3} \mcE_u^{1/2} \mcE_1 \leq C_B \tilde{\Xi} \tilde{\mcE}_u^{1/2}\tilde{\mcE}_3.
\end{align*}

\subsection{The new stress from solving divergence equations} \label{subsec:NewStressDivEq}
The part of the new stress consisting of $R_H + R_T + R_{L}$ will be estimated by solving divergence equations of the form
\begin{align*}
\nabla \cdot R_I = e^{i \lambda \xi_I}h_I,
\end{align*}
with $h_I \in C_c^{\infty}(Q_I)$ compactly supported in $Q_I=Q_{\kappa(I)}$. To this end, we adapt the method in \cite{IsettOh16} to solve partially symmetric divergence equations with compactly supported sources. Heuristically one can obtain a solution $R_I \in C_c^{\infty}(\hat{Q}_I)$ such that $\|R_I\| \sim \lambda^{-1}\|h_I\|$ with a slightly enlarged support $\hat{Q}_I$. The precise statements are contained in the following lemma.

Given a smooth vector field $\bar{U}=(\bar{U}^1,\cdots,\bar{U}^d)(t,z^1,\cdots,z^d)$, a set of positive numbers $\bar{\tau}, \bar{\ell}_1, \cdots, \bar{\ell}_d$, and a point $(t_0,z_0) \in \R \times \R^d$, the Eulerian cylinders convected by the flow of $\bar{U}$ is defined as in \cite{IsettOh16}:
\begin{align}
\hat{Q}_{\bar{U}}(\bar{\tau}, \bar{\ell}_1, \cdots, \bar{\ell}_d; (t_0,z_0)) = \{(t,z): |t-t_0| \leq \bar{\tau}, |z^i - \Phi_{t-t_0}^i(t_0,z_0)| \leq \bar{\ell}_i \}, \label{def:hatQ}
\end{align}
where $\Phi$ is the flow generated by $(\ptl_t + \bar{U} \cdot\nabla_z)$ defined in \eqref{def:Phi_s}. Setting as in \eqref{set:flow-scaling}, one gets from
Lemma \ref{Lemma:flow-scaling} that, for  $|s| \leq \tau, (x,y) \in \tilde{Q}_{\tilde{\kappa(I)}}$,
\begin{align*}
|\Phi_s^i(t_I,x,y) - \Phi_s^i(t_I,x_I,y_I)|
&\leq   \begin{cases}
A_2 e^{A_1} \ell_x \leq 3\ell_x, \quad i=1,2,\\
A_2 e^{A_1} \ell_y \leq 3\ell_y, \quad i = 3,
\end{cases}
\end{align*}
where $(t_I,x_I,y_I)=q_{\kappa(I)}$ for $q_{\kappa(I)}$  in  \eqref{def:p_kappa}.
Hence from the definitions of $Q_{\kappa}$ in \eqref{def:p_kappa}, one has
\begin{align}
Q_{\kappa} \subset \hat{Q}_{\kappa } := \hat{Q}_{u_\ell,v_\ell}(\tau,3 \ell_x, 3\ell_x, 3\ell_y;q_{\kappa}). \label{supp-Q_k-hat-Q_k}
\end{align}

\begin{Lemma}\label{Lemma:SolveSymmDiv}
	Suppose that $(H^1,H^2)(t,x,y) = e^{i \lambda \xi(t,x)}(h^1,h^2)(t,x,y) \in C_c^{\infty}(\hat{Q})$ are smooth functions with supports in $\hat{Q}=\hat{Q}_{u_\ell,v_\ell}(\tau,3 \ell_x,3 \ell_x,3 \ell_y;q_{\kappa}) \subset \mfD$  such that $\nabla_x \xi(t,x)$ is a constant and satisfies the estimate
	\begin{align}
	\|\barDt \xi\|_{L^{\infty}(\supp H)} \leq C_0  B^{-1}N^{-1/3}\mcE_u^{1/2}, \quad 1 \leq |\nabla_x \xi| \leq 100. \label{est-xi}
	\end{align}
	Moreover, $H=(H^1,H^2)$ satisfies the compatibility conditions
	\begin{align}
	\int H^l dxdy = 0, \quad  \int x^j H^l - x^l H^jdxdy = 0, \quad j,l=1,2, \label{H^l-conditions-3}
	\end{align}
	and the estimates that, for $0 \leq |\beta + \gamma | \leq 1, |\alpha| \geq 0$,
	\begin{align}
	\| (\barDt)^{\gamma}  \nabla_y^{\beta} \nabla_x^{\alpha} h \| & \leq C_{\alpha,\beta,\gamma} \tau^{-\gamma}\ell_y^{-\beta}\ell_x^{-\alpha}  B^2 N^{2/3}\Xi \mcE_u^{1/2}\mcE_1^{1/2}. \label{est-h_I-y}
	\end{align}
	Then there exist two constants  $C_1 = C_1(\vartheta,C_0,C_{\alpha,\alpha',\beta,\gamma}), C_2 =  C_2(B,\vartheta,C_0,C_{\alpha,\alpha',\beta,\gamma})$, and a $2 \times 3$ matrix $T=T^{kl} \in C_c^{\infty}(\hat{Q})$ which solves the equations
	\begin{align}
	\sum_{j=1}^2 \frac{\ptl }{\ptl x^j} T^{jl} + \frac{\ptl }{\ptl y} T^{3l} = e^{i \lambda \xi}h^l, \quad l = 1,2, \label{eq:div} \\
	T^{jl} = T^{lj}, \quad j,l = 1, 2,
	\end{align}
	with
	\begin{align}
	\supp T \subset \hat{Q} = \hat{Q}_{u_\ell,v_\ell}(\tau,3 \ell_x, 3 \ell_x, 3 \ell_y;q_{\kappa}). \label{supp-T}
	\end{align}
	Furthermore, $T$ satisfies the following estimates
	\begin{equation}
	\begin{aligned}
	\|T\| & \leq C_1 B^{-1} N^{-1/3}\mcE_u^{1/2}\mcE_1^{1/2}, \quad
	\|\nabla_x T\| \leq C_2 N^{2/3}\Xi\mcE_u^{1/2}\mcE_1^{1/2}, \\
	\|\ptl_y T\| &\leq  C_2 N^{1/6}\Xi^{1/2}\mcE_u^{1/2}\mcE_1^{3/4}, \quad
	\|\barDt T \| \leq C_2  N^{2/3} \Xi \mcE_u^{1/2}\mcE_1. \label{est-R-div}
	\end{aligned}
	\end{equation}
\end{Lemma}

\begin{Remark}\label{Remark:Compatibility}
	Using integration by parts, one can verify directly that if $H$ is of the form $H^l = \sum_{j=1}^{2}\frac{\ptl}{\ptl x^j} S^{jl} + \ptl_y Y^l$, where  $S$ is a symmetric $2 \times 2$ matrix, $Y$ is a $2$-vector, and $\supp (S,Y) \subset \hat{Q}$, then $H$ satisfies the compatibility conditions \eqref{H^l-conditions-3}.
\end{Remark}

\subsubsection{Solving divergence equations with symmetry}\label{subsec:SolveDiv}

The following result for solving symmetric divergence equations is an anisotropic variant of \cite[Theorem 11.1]{IsettOh16}.
\begin{Lemma}\label{Lemma:OperatorSymmDiv}
	Let $\bar{\tau}, \bar{\ell}_x, \bar{\ell}_y, \Lambda_t, \Lambda_x, \Lambda_y, \delta_U, \delta_H $ be given positive constants that satisfy
	\begin{align}
	\bar{\tau}   &\leq \frac{\min \{ \bar{\ell}_x, \bar{\ell}_y \}}{\delta_U}, \quad \Lambda_t \geq \frac{\delta_U}{\min \{ \bar{\ell}_x, \bar{\ell}_y \}}. \label{ineq:bartau}
	\end{align}
	Suppose that $\bar{U}=(\bar{U}^1,\bar{U}^2,\bar{U}^{3})$ is a smooth vector field on $\mathbb{R} \times \mathbb{R}^2 \times \mathbb{R}$ with
	\begin{align}
	\| \nabla_x^{\alpha} \ptl_y^{\beta}   \bar{U} \|_{L^{\infty}(\hat{Q})} &\leq C_{0} \bar{\ell}_x^{-|\alpha|} \bar{\ell}_y^{-\beta} \delta_U, \quad 0 \leq |\alpha|+\beta \leq 1, \label{est-barU}
	\end{align}
	where for some point $q_0=(t_0,x_0,y_0) \in \mathbb{R} \times \mathbb{R}^2 \times \mathbb{R}$, we denote
	\begin{align*}
	\hat{Q}=\hat{Q}_{\bar{U}}(\bar{\tau},\bar{\ell}_x,\bar{\ell}_x,\bar{\ell}_y;(t_0,x_0,y_0)).
	\end{align*}
	Let $H=(H^1,H^2)(t,x,y) \in C_c^{\infty}(\hat{Q})$
	satisfy the conditions \eqref{H^l-conditions-3} and the estimates
	\begin{align}
	\|(\ptl_t + \bar{U} \cdot \nabla_{x,y})^{\gamma}\ptl_y^{\beta}\nabla_x^{\alpha}  H\|_{C^0} \leq C_0  \Lambda_x^{|\alpha|} \Lambda_y^{\beta} \Lambda_t^{\gamma} \delta_H, \quad 0 \leq |\alpha|+\beta+\gamma \leq 1. \label{est-H}
	\end{align}
	Then there exist a $2 \times 3$ matrix $R^{jl}[H] \in C_c^{\infty}(\hat{Q})$ which solves
	\begin{align}
	\sum_{j=1}^2 \frac{\ptl }{\ptl x^j} R^{jl} + \frac{\ptl }{\ptl y} R^{3l} = H^l, \quad l = 1,2 \label{eq:SymmDiv},\\
	R^{jl} = R^{lj}, \quad j,l = 1, 2,
	\end{align}
	depending linearly on $H$, and a constant $C=C(C_0)$ such that
	\begin{equation*}
	\begin{aligned}
	&\|\ptl_y^{\beta}\nabla_x^{\alpha} R^{kl}\|_{C^0} \leq C  \sum_{m=0}^{|\alpha|} \bar{\ell}_x^{-(|\alpha|-m)} \Lambda_x^m \sum_{j=0}^{\beta} \bar{\ell}_y^{-(|\beta|-j)} \Lambda_y^{j} \max \{ \bar{\ell}_x, \bar{\ell}_y \} \delta_H, \quad 0 \leq |\alpha|+\beta \leq 1, \\
	&\| (\ptl_t + \bar{U} \cdot \nabla_{x,y})R^{jl} \|_{C^0}  \leq C  \Lambda_t \max \{ \bar{\ell}_x, \bar{\ell}_y \} \delta_H.
	\end{aligned}
	\end{equation*}
\end{Lemma}

We postpone the proof of Lemma \ref{Lemma:OperatorSymmDiv} to Appendix \ref{sec:pf-op-symm}.
Using Lemma \ref{Lemma:OperatorSymmDiv}, one can prove Lemma \ref{Lemma:SolveSymmDiv} by following the approach in \cite{Isett12,IsettOh16}.
\begin{proof}[Proof of Lemma \ref{Lemma:SolveSymmDiv}]
	Following \cite{Isett12,IsettOh16},
	one can construct the solution $T$ as the sum of an approximate solution $T_{(D)}$ and a correction $R$:
	\begin{align}
	T^{jl} &= T_{(D)}^{jl} + R^{jl}, \quad
	T_{(D)}^{jl} =  \sum_{k=1}^{D}\frac{1}{\lambda}(e^{i\lambda \xi}q_{(k)}^{jl}), \quad j,l=1,2 \label{eq:T^jl}.
	\end{align}
	Here $D$ is the smallest integer such that
	\begin{align}
	\frac{1}{3}(2D-3) \geq \vartheta^{-1}, \label{def:D}
	\end{align}
	where $\vartheta$ is the constant of Lemma \ref{Lemma:Iteration} in \eqref{ineq:N}.
	The amplitudes $q_{(k)}^{jl}$ are symmetric matrices obtained by solving the following linear equations, for $k=1,\cdots,D$:
	\begin{align}
	\sum_{j=1}^{2} i \frac{\ptl \xi}{\ptl x^j}   q_{(k)}^{jl} = h_{(k)}^l, \quad h_{(1)}^l = h^l, \quad h_{(k+1)}^l = -\frac{1}{\lambda}\sum_{j=1}^{2}\frac{\ptl }{\ptl x^j} q_{(k)}^{jl}. \label{eq:q_k}
	\end{align}
	Set, for  $k=1,\cdots,D$:
	\begin{align}
	q_{(k)} = \frac{-i}{|\nabla_x \xi|}\big( (h_{(k)} \cdot \vec{e}_{\parallel}) I + (h_{(k)} \cdot \vec{e}_{\perp})(\vec{e}_{\parallel} \otimes \vec{e}_{\perp} + \vec{e}_{\perp} \otimes \vec{e}_{\parallel}) \big), \label{def:q_k}\\
	\vec{e}_{\parallel}  = \frac{\nabla_x \xi}{|\nabla_x \xi|}, \quad \vec{e}_{\perp} = \frac{(\nabla_x \xi)^{\perp}}{|(\nabla_x \xi)^{\perp}|}. \nonumber
	\end{align}
	It is straightforward to verify that $q_{(k)}^{jl}$ is symmetric in $jl$, and solves \eqref{eq:q_k} with $\supp q_{(k)} \subset \supp H \subset \hat{Q}$. The definitions \eqref{eq:q_k} and \eqref{def:q_k} imply that
	\begin{align*}
	h_{(k+1)} &=  \frac{i}{\lambda|\nabla_x \xi|} \big(  \nabla_xh_{(k)} \cdot \vec{e}_{\parallel}  + (\vec{e}_{\perp}^T \nabla_x h_{(k)}  \vec{e}_{\perp})\vec{e}_{\parallel}  + (\vec{e}_{\parallel}^T \nabla_x h_{(k)}  \vec{e}_{\perp}) \vec{e}_{\perp}   \big)\\
	& = \frac{1}{\lambda} C_{(k+1)}  \nabla_x h_{(k)}, \quad  k = 1,\cdots,D,
	\end{align*}
	for some constant tensors $C_{(k+1)}$. Thus,
	\begin{align}
	h_{(k)} = \lambda^{-(k-1)} C_{(k)} \cdots C_{(2)} (\nabla_x)^{k-1} h, \quad  k = 2,\cdots,D+1. \label{eq:h_(k)}
	\end{align}
	\eqref{est-h_I-y} shows that, for $0 \leq |\beta + \gamma | \leq 1, |\alpha| \geq 0$,
	\begin{align*}
	| (\barDt)^{\gamma}  \nabla_y^{\beta} \nabla_x^{\alpha} q_{(k)} \| & \lesssim \lambda^{-(k-1)}\| (\barDt)^{\gamma}  \nabla_y^{\beta} \nabla_x^{\alpha+k-1}  h \|  \\
	&\leq C_{\alpha,\beta,\gamma} \tau^{-\gamma}\ell_y^{-\beta}\ell_x^{-\alpha} (\lambda \ell_x)^{-(k-1)} B^2 N^{2/3}\Xi \mcE_u^{1/2}\mcE_1^{1/2}.
	\end{align*}
	This and the relation \eqref{ineq:N-Cor} imply that
	\begin{align*}
	\|T_{(D)}\| &\leq \lambda^{-1}\sum_{k=1}^{D}\| q_{(k)} \| \lesssim B^{-1} N^{-1/3}\mcE_u^{1/2}\mcE_1^{1/2},\\
	\|\nabla_x T_{(D)}\| &\leq  \lambda^{-1}\sum_{k=1}^{D} (\lambda \|q_{(k)} \| + \| \nabla_x q_{(k)} \|)\leq C_B N^{2/3}\Xi\mcE_u^{1/2}\mcE_1^{1/2}, \\
	\|\ptl_y T_{(D)}\| &\leq\lambda^{-1}\sum_{k=1}^{D}\| \ptl_y q_{(k)} \| \leq  C_B \Xi^{1/2}\mcE_u^{3/4}\mcE_1^{1/2} \leq C_B N^{1/6}\Xi^{1/2}\mcE_u^{1/2}\mcE_1^{3/4}.
	\end{align*}
	It follows from \eqref{est-xi} and \eqref{ineq:N-Cor} that
	\begin{align*}
	\|\barDt T_{(D)} \| &\lesssim \sum_{k=1}^{D} (\|\barDt\xi\|\| q_{(k)}\| + \lambda^{-1}\| \barDt q_{(k)}\|)  \leq C_B N^{1/3}\Xi\mcE_u\mcE_1^{1/2} \\
	&\leq C_B  N^{2/3} \Xi \mcE_u^{1/2}\mcE_1.
	\end{align*}
	Hence the estimates \eqref{est-R-div} hold for $T_{(D)}$.
	
	Note that  \eqref{eq:q_k} implies that, for $k=1,\cdots,D$:
	\begin{align}
	H_{(k+1)}^l :=  e^{i\lambda \xi}h_{(k+1)}^l =  H^l -  \sum_{j=1}^{2}\frac{\ptl }{\ptl x^j} \sum_{m=1}^{k}\frac{1}{\lambda}(e^{i\lambda \xi}q_{(m)}^{jl}). \label{eq:H_{k+1}}
	\end{align}
	Due to \eqref{eq:div} and \eqref{eq:T^jl},
	$R$ satisfies the divergence equations:
	\begin{equation}
	\begin{aligned}
	\sum_{j=1}^{2}\frac{\ptl }{\ptl x^j} R^{jl} +  \frac{\ptl }{\ptl y} R^{3l} = H_{D+1}^l = e^{i\lambda \xi}h_{(D+1)}^l, \label{eq:R}\\
	R^{jl} = R^{lj}, \quad j,l = 1, \cdots, 2.
	\end{aligned}
	\end{equation}
	It follows from \eqref{ineq:N-Cor}, \eqref{est-xi}, \eqref{est-h_I-y},  and \eqref{eq:h_(k)} that for $0 \leq |\alpha|+\beta+\gamma \leq 1$,
	\begin{align*}
	\| \barDt^{\gamma} \ptl_y^{\beta} \nabla_x^{\alpha} H_{(D+1)} \|
	&\lesssim (B^{2}N^{2/3}\Xi\mcE_u^{1/2})^{\gamma} \ell_y^{-\beta} \lambda^{\alpha}  (\lambda \ell_x)^{-D}  B^2 N^{2/3}\Xi \mcE_u^{1/2}\mcE_1^{1/2}\\
	& \lesssim (\lambda \mcE_1^{1/2})^{\gamma} \ell_y^{-\beta} \lambda^{\alpha} (B^2N^{2/3})^{-D} B^2 N^{2/3}\Xi \mcE_u^{1/2}\mcE_1^{1/2}\\
	&\leq C (\lambda \mcE_1^{1/2})^{\gamma} \ell_y^{-\beta} \lambda^{\alpha}  B^{-2(D-1)}N^{-1/3}  \mcE_u^{1/2}\mcE_1^{1/2},
	\end{align*}
	where \eqref{ineq:N} and \eqref{def:D} have been used for the last inequality.
	It follows from \eqref{eq:H_{k+1}} and Remark \ref{Remark:Compatibility} that $H_{(D+1)}^l$ also satisfies the  compatibility conditions \eqref{H^l-conditions-3}. In view of \eqref{est-D_x-u_ell} and the above estimates, one can apply Lemma \ref{Lemma:OperatorSymmDiv} with
	\begin{align*}
	q_0 = q_{\kappa}, \quad \bar{\tau} = \tau, \bar{\ell}_x = 3 \ell_x, \bar{\ell}_y = 3 \ell_y, \quad \Lambda_t = \lambda \mcE_1^{1/2}, \Lambda_x = \lambda, \Lambda_y = \ell_y^{-1}, \\
	H = H_{(D+1)}, \ \bar{U} = (u_{\ell},v_{\ell}), \quad
	\delta_U = \mcE_u^{1/2}, \quad \delta_H =  B^{-2(D-1)}N^{-1/3}  \mcE_u^{1/2}\mcE_1^{1/2}.
	\end{align*}
	to obtain a solution $R^{jl} \in C_c^{\infty}(\hat{Q})$ to \eqref{eq:R} with the estimates \eqref{est-R-div}.
\end{proof}

\subsubsection{Verifications of the assumptions}		
Note that one can write
\begin{align}
\nabla \cdot (R_H + R_T + R_L) &=   \sum_I \ptl_t (e^{i \lambda \xi_I} \tilde{U}_I) + \sum_I \nabla \cdot (S_I,Y_I) + \sum_{J \neq \bar{I}} \nabla \cdot (S_{IJ},Y_{IJ}), \label{eq:R_H-R_T-R_L}\\
(S_{IJ}, Y_{IJ}) &= e^{i \lambda (\xi_I+\xi_J)}( \tilde{U}_I \otimes \tilde{U}_J,\tilde{V}_I  \tilde{U}_J), \nonumber \\
(S_{I},Y_{I}) &=  e^{i \lambda \xi_I}(u_{\ell} \otimes \tilde{U}_I + \tilde{U}_I \otimes u_{\ell}, v_{\ell} \tilde{U}_I - \ptl_y \tilde{U}_I + \tilde{V}_I u_{\ell}). \nonumber
\end{align}
The individual terms in \eqref{eq:R_H-R_T-R_L} are either of the form $e^{i\lambda \xi_I}h_I$ or $e^{i\lambda (\xi_I+\xi_J)}h_{IJ}, J \neq \bar{I}$, with supports in $Q_{\kappa(I)} \subset \hat{Q}_{\kappa(I)}$ (recall \eqref{supp-Q_k-hat-Q_k}). It follows from \eqref{supp-w_I-Q_I} that $h_I = 0 = h_{IJ}$ for any $Q_{\kappa(I)} \nsubseteq \mfD$.
If the assumptions in Lemma \ref{Lemma:SolveSymmDiv} are verified for the terms  supported in $Q_{\kappa(I)} \subset \mfD$,  one can obtain a solution $T_I$ supported in $\hat{Q}_{\kappa(I)}$ with the estimates \eqref{est-R-div}, which would imply  \eqref{tilde-FE-levels} in Lemma \ref{Lemma:Iteration}.

We first verify the estimates \eqref{est-xi} for $\xi_I$ and $\xi_{IJ}:=\xi_I + \xi_J, J \neq \bar{I}$. It is clear from the definitions \eqref{def:xi_I} that $\nabla_x \xi_I$  are constants with $1 \leq |\xi_I| = [\kappa(I)] \leq 2^4=16$.
Recall that $\supp h_I \cup \supp h_{IJ} \subset Q_I$. It follows from \eqref{eq:xi_I} and \eqref{ineq:u_ell-u_I} that
\begin{align}
\|\barDt \xi_I\|_{L^{\infty}(Q_I)} = \|u_{\ell}-u_I\|_{L^{\infty}(Q_I)} |\nabla_x \xi_I| \leq C B^{-1}N^{-1/3}\mcE_u^{1/2}. \label{est-Dt-xi}
\end{align}
Thus $\xi_I$ satisfies \eqref{est-xi}.
Due to \eqref{supp-eta-psi_k}, $h_{IJ} = 0$ if $Q_{\kappa(I)} \cap Q_{\kappa(J)} = \emptyset$. It follows from \eqref{ineq:xi_I-xi_J} that $\xi_{IJ}$ also satisfies the estimates \eqref{est-xi} for any indices $I,J$ such that $Q_{\kappa(I)} \cap Q_{\kappa(J)} \neq \emptyset$ and $J \neq \bar{I}$.

Due to Remark \ref{Remark:Compatibility}, the terms $\sum_I \nabla \cdot (S_I,Y_I) + \sum_{J \neq \bar{I}} \nabla \cdot (S_{IJ},Y_{IJ})$  satisfy the compatibility conditions \eqref{H^l-conditions-3}. It follows from  \eqref{def:w_I} that
\begin{align*}
e^{i \lambda \xi_I} \tilde{U}_I = -\Delta (\frac{1}{\lambda^2|\nabla \xi_I|^2}e^{i\lambda\xi_I}U_I) + \nabla_x \nabla \cdot (\frac{1}{\lambda^2|\nabla \xi_I|^2}e^{i\lambda\xi_I}W_I).
\end{align*}
Using integrations by parts, one has, for any $t \in \Rp$,
\begin{align*}
\int e^{i \lambda \xi_I} \tilde{U}_I^l dxdy  = 0, \quad  \int  x^j e^{i \lambda \xi_I} \tilde{U}_I^l dx dy = 0, \quad j,l=1,2.
\end{align*}
This shows that $\ptl_t (e^{i \lambda \xi_I} \tilde{U}_I)$ also satisfies  \eqref{H^l-conditions-3}.

It remains to verify the estimates \eqref{est-h_I-y} for the terms in \eqref{eq:R_H-R_T-R_L}, which are given in the following subsections.

\subsubsection{The high-high interactions terms $R_H$}\label{sec:R_H}
Recall from \eqref{eq:R_H} that
\begin{align*}
\nabla \cdot R_H = \sum_{J \neq \bar{I}} \nabla_x \cdot (e^{i \lambda (\xi_I+\xi_J)} \tilde{U}_I \otimes \tilde{U}_J) +  \sum_{J \neq \bar{I}}  \ptl_y (e^{i \lambda (\xi_I+\xi_J)}  \tilde{V}_I  \tilde{U}_J).
\end{align*}
It follows from the definitions \eqref{def:U-V} and \eqref{def:xi_I} that, for any index $I$ and $J$,
\begin{align}
\nabla_x \xi_I \cdot U_J =  (-1)^{s_I}[\kappa(I)]\f_1^{\perp} \cdot \eta_J \psi_J a_J \f_1  =0. \label{eq:xi_I-U_J}
\end{align}
Thus,
\begin{align*}
 &\nabla_x \cdot(e^{i \lambda (\xi_I+\xi_J)} \tilde{U}_I \otimes \tilde{U}_J)
=e^{i\lambda (\xi_I + \xi_J)}(i\lambda \nabla_x (\xi_I+\xi_J) \cdot (U_I + \delta U_I)  \tilde{U}_J + \nabla_x \cdot (\tilde{U}_I \otimes \tilde{U}_J))\\
&=e^{i\lambda (\xi_I + \xi_J)}\big(i\lambda \nabla_x (\xi_I+\xi_J) \cdot \delta U_I \tilde{U}_J + (\nabla_x \cdot \tilde{U}_I) \tilde{U}_J + (\tilde{U}_I \cdot \nabla_x)\tilde{U}_J \big) := e^{i\lambda (\xi_I + \xi_J)}h_{H,IJ,(1)}.
\end{align*}
It follows from  \eqref{def:ell}, \eqref{est-W_I-Dt} and \eqref{est-deltaW_I-Dt}   that $h_{H,IJ,(1)}$  satisfies \eqref{est-h_I-y}.
Since $\xi_I=\xi(t,x)$ is independent of $y$ due to \eqref{def:xi_I}, one has
\begin{align*}
  \ptl_y (e^{i \lambda (\xi_I+\xi_J)}  \tilde{V}_I  \tilde{U}_J) = e^{i \lambda (\xi_I+\xi_J)}  (\ptl_y \tilde{V}_I  \tilde{U}_J +  \tilde{V}_I  \ptl_y \tilde{U}_J) :=  e^{i \lambda (\xi_I+\xi_J)}h_{H,IJ,(2)}.
\end{align*}
Consequently, \eqref{eq:ell},  \eqref{def:ell} and \eqref{est-W_I-Dt} show that $h_{Y,IJ,(2)}$ also satisfies \eqref{est-h_I-y}.

\subsubsection{The high-low interactions terms $R_L$} Due to \eqref{eq:w_I-div-free}, $\nabla_x \cdot U + \ptl_y V =  \sum_I \nabla \cdot w_I = 0$, thus
\begin{align*}
\nabla \cdot R_{L}& =  \nabla_x \cdot ( U \otimes u_{\ell}) + \ptl_y (V u_{\ell}) = (U \cdot \nabla_x) u_{\ell} + V \ptl_y u_{\ell} \\
&= \sum_I e^{i \lambda \xi_I}( (\tilde{U}_I \cdot \nabla_x)u_{\ell} + \tilde{V}_I \ptl_y u_{\ell}  ) := \sum_I e^{i \lambda \xi_I} h_{L,I}.
\end{align*}
It follows from \eqref{eq:ell}, \eqref{est-D_x-u_ell}, \eqref{est-D_t-u_ell}, and \eqref{est-W_I-Dt} that $ h_{L,I}$ satisfies \eqref{est-h_I-y}.

\subsubsection{The transport-diffusion terms $R_T$}
Since $\nabla_x \cdot u_{\ell} + \ptl_y v_{\ell} =0$ in $\mfD$ due to  \eqref{eq:div-free-u-v-ell} and $\supp U \subset \mfD$,  one gets from  \eqref{eq:Dt-W_I} that
\begin{align*}
&\nabla \cdot R_T = \ptl_t U + \nabla_x \cdot (u_{\ell} \otimes U) + \ptl_y(v_{\ell} U) - \ptl_{yy}^2 U  = (\barDt - \ptl_{yy}^2) U \\
&= \sum_I \barDt(e^{i\lambda \xi_I}\tilde{U}_I) -\sum_I  e^{i\lambda \xi_I}\ptl_{yy}^2 \tilde{U}_I = \sum_I  e^{i\lambda \xi_I}(i \lambda \tilde{U}_I (u_{\ell}-u_I) \cdot \nabla_x \xi_I + \barDt \tilde{U}_I - \ptl_{yy}^2 \tilde{U}_I ) \\
&:=\sum_I  e^{i\lambda \xi_I}(h_{T,I,(1)} + h_{T,I,(2)} + h_{T,I,(3)}).
\end{align*}
It follows from \eqref{est-D_x-u_ell},  \eqref{est-D_t-u_ell}, \eqref{est-W_I-Dt} and  \eqref{ineq:u_ell-u_I} that $h_{T,I,(1)}$ satisfies \eqref{est-h_I-y}. While that $h_{T,I,(2)}$ and $h_{T,I,(3)}$ satisfy \eqref{est-h_I-y} follows easily from  \eqref{est-W_I-Dt}.

\subsubsection{Conclusion of the proof of Lemma \ref{Lemma:Iteration}}
It follows from \eqref{eq:R_S} and \eqref{eq:R_M} that
\begin{align*}
\big(\text{supp}_{t,y} R_M \cup \text{supp}_{t,y} R_S \big)  \subset \big(\text{supp}_{t,y} R \cup \text{supp}_{t,y} (S_{\ell}, Y_{\ell}) \cup \supp e \cup \supp e_{\ell} \cup \text{supp}_{t,y} w \big).
\end{align*}
Recall from \eqref{eq:R_H-R_T-R_L} that  $R_H, R_L$, and $R_T$ are obtained by solving divergence equations of the form $\nabla \cdot T_I = h_I$, with $\supp h_I \subset Q_I \subset \supp w$.
Thus, \eqref{supp-T} implies that
\begin{align*}
\big( \text{supp}_{t,y} R_T \cup \text{supp}_{t,y} R_H \cup \text{supp}_{t,y} R_L \big) \subset N(\text{supp}_{t,y} w; \tau, 3 \ell_y).
\end{align*}
As a consequence of \eqref{def:tlR}, \eqref{supp-e}, \eqref{supp-S_ell-Y_ell}, \eqref{supp-e_ell}, and  \eqref{supp-w}, it holds that
\begin{align*}
\text{supp}_{t,y} \delta R \subset N(\supp \e;\frac{1}{2}\ell^2,\frac{1}{2}\ell) \cup N(\text{supp}_{t,y} w; \tau, 3 \ell_y) \subset N(\supp \e;\ell^2,\ell),
\end{align*}
if one chooses the constant $B \geq 8$.
Therefore,
\begin{align*}
\text{supp}_{t,y} \tilde{R} &\subset \big(\text{supp}_{t,y} R \cup \text{supp}_{t,y} \delta R \big) \subset N(\supp \e;\ell^2,\ell).
\end{align*}
Together with \eqref{supp-w},  this  yields the desired estimates \eqref{supp-w-tildeR} for the supports.

Collecting all the estimates above shows that there exists a constant $C_1=C_1(\vartheta)$ which is  independent of $B$ and $N$ such that
\begin{align*}
\|\delta R\|_{C^0} = \|R_S + R_{M}  + R_H + R_T + R_L \|_{C^0}  \leq C_1 B^{-1} N^{-1/3}\mcE_u^{1/2}\mcE_1^{1/2}.
\end{align*}
Now one can fix the constant $B$ such that
\begin{align*}
B \geq \max \{ C_1, 8 \}.
\end{align*}
It follows from the estimates above and the relations \eqref{ineq:N} that there exists a constant $C_{\vartheta}$  such that $(\tilde{u},\tilde{v},\tilde{R})$ have frequency-energy levels below  $(\tilde{\Xi},\tilde{\mcE})$ as given by \eqref{tilde-FE-levels}.
The proof of Lemma \ref{Lemma:Iteration} is completed.

\section{Proof of the theorems}\label{sec:iterThm}

\subsection{Proof of Theorem \ref{Thm:Prandtl}}\label{subserc:proofPrandtl}

\subsubsection{Setting up}

Let $(\underline{u},\underline{v})$ be given as in Theorem \ref{Thm:Prandtl}. Starting with $(u_{(0)},v_{(0)})=(\underline{u},\underline{v})$, we will use Lemma \ref{Lemma:Iteration} to construct iteratively a sequence of solutions $(u_{(n)},v_{(n)},R_{(n)})$ to the approximate system \eqref{eq:PrStress}, with frequency-energy levels below
\begin{align}
(\Xi_{(n)},\mcE_{(n)})=(\Xi_{(n)},\mcE_{(n),u},\mcE_{(n),1}, \mcE_{(n),2}, \mcE_{(n),3}).
\end{align}

Let $\delta, \eps \in (0,1)$ be positive constants to be chosen as follows.  Given two constants $\alpha_0 \in  (0,1/21)$ and $\beta_0 \in (0,1/10)$, we choose the constant $\delta \in (0,1)$ sufficiently small so that
\begin{equation}
\alpha_0 < \frac{1}{21+6\delta(\delta^2+4\delta+6)}, \quad \beta_0 < \frac{1}{10+3\delta(\delta^2+4\delta+6)}. \label{ineq:alpha}
\end{equation}
Set
\begin{equation}
\vartheta = \min\{\frac{\delta}{40},\frac{\log 2}{4 \log 10} \}, \label{def:vartheta} \end{equation}
and let $C=C_{\vartheta}$ be the constant in Lemma  \ref{Lemma:Iteration}.
Suppose that $H = (H^1,H^2)$ is a smooth vector field on $\mathbb{T}^2$ with $\int_{\mathbb{T}^2} H dx = 0$. Let $\Delta_x^{-1} H$ be the solution to the Poisson equation on $\mathbb{T}^2$ with zero averages:
\begin{align*}
\Delta_x (\Delta_x^{-1} H) = H, \quad  \int_{\mathbb{T}^2} \Delta_x^{-1} H dx = 0.
\end{align*}
Define
\begin{align*}
\mathcal{R}[H] = (\nabla_x \cdot \Delta_x^{-1} H) I + (\nabla_x \mathcal{P} \Delta_x^{-1} H + (\nabla_x \mathcal{P} \Delta_x^{-1} H)^t ),
\end{align*}
where $\mathcal{P} = I - \nabla_x \Delta_x^{-1} \nabla_x \cdot$ is the projection into divergence-free vector field on $\mathbb{T}^2$.
It is straightforward to verify that $\mathcal{R}[H]$ is a symmetric $2 \times 2$ matrix and solves
\[ \nabla_x \cdot \mathcal{R}[H] = H. \]
It follows from \eqref{u-u_S} that the following mean-zero conditions are satisfied:
\begin{align*}
\int_{\T^2} \ptl_t (\underline{u} - u_C)dx = 0, \quad \int \ptl_{yy}^2 (\underline{u} - u_C) dx  = 0, \quad \text{ in } \Rp \times \Rp.
\end{align*}
Set $(u_{(0)},v_{(0)}) =(\underline{u},\underline{v}),$ and
\begin{align*}
R_{(0)} &= \Big(\underline{u} \otimes \underline{u} - u_C \otimes u_C, \underline{v} \underline{u} - v_C u_C\Big) + \Big(\mathcal{R}(\ptl_t (\underline{u} - u_C) - \ptl_{yy}^2 (\underline{u} - u_C)),0\Big).
\end{align*}
Due to the assumptions \eqref{eq:div-bu-bv} and  that $(u_C,v_C)$ is a classical solution to the system \eqref{eq:Prandtl}, it is straightforward to verify that $(u_{(0)},v_{(0)},R_{(0)})$ solves the approximate system \eqref{eq:PrStress} with  $\text{supp}_{t,y} R_{(0)} \subset \text{supp}_{t,y}(\underline{u} - u_C,\underline{v}-v_C)$.

Set
\begin{equation}
\mcD_{(0)} = \text{supp}_{t,y}(\underline{u} - u_C,\underline{v}-v_C), \quad
\tilde{\mcD} = N(\mcD_{(0)};\rho,\rho^{1/2}). \label{def:tildemcD}
\end{equation}
Let
\begin{align}
\mcE_{(0),u} &= 25 \bar{\mcE}, \quad \mcE_{(0),1} = 25\bar{\mcE}, \quad \mcE_{(0),2} =2 \bar{\mcE}, \quad \mcE_{(0),3} = \bar{\mcE}, \label{eq:E_(0)}
\end{align}
where
\begin{equation}
\bar{\mcE} = \max\{\|(u_{(0)},v_{(0)})\|_{C^0}^2,\|R_{(0)}\|_{C^0},1\}.
\end{equation}
Fix $\Xi_{(0)}$ to be a large constant such that
\begin{align}
\Xi_{(0)} \geq 10000 (\rho^{-1}+1)\bar{\mcE}^{3/2}C_{\vartheta}^2 \label{ineq:Xi_0},
\end{align}
and for $0 \leq |\alpha|+\beta/2+\gamma\leq 1$,
\begin{align*}
\|(\ptl_t + u_{(0)} \cdot \nabla_x + v_{(0)} \ptl_y )^{\gamma}  \nabla_{x}^{\alpha} \ptl_y^{\beta}  R_{(0)}\|_{C^0} &\leq \Xi_{(0)}^{|\alpha|} (\Xi_{(0)}^{1/2}\bar{\mcE}^{1/4})^{\beta+2\gamma} \bar{\mcE},\\
\|(\ptl_t + u_{(0)} \cdot \nabla_x + v_{(0)} \ptl_y )^{\gamma}  \nabla_{x}^{\alpha} \ptl_y^{\beta}  (u_{(0)},v_{(0)})\|_{C^0} &\leq \Xi_{(0)}^{|\alpha|} (\Xi_{(0)}^{1/2}\bar{\mcE}^{1/4})^{\beta+2\gamma} \bar{\mcE}^{1/2}.
\end{align*}
Then it is direct to verify that $(u_{(0)},v_{(0)},R_{(0)})$ has frequency-energy levels below $(\Xi_{(0)},\mcE_{(0)})$.
Choose $\eps$ sufficiently small so that
\begin{equation}
\begin{aligned}
\eps \leq \max\{ C_{\vartheta}^{-1/2}, (2^6C_{\vartheta})^{-\vartheta/\delta - \vartheta}, \frac{1}{4}\bar{\mcE},  \bar{\mcE}^{-1}, \Xi_{(0)}^{-1/4} \} \label{ineq:eps-1},
\end{aligned}
\end{equation}
where
\begin{align}
b=3(1+\delta)^4-\frac{3}{2}(2+\delta) = 3\delta(\frac{7}{2} + 6\delta + 4\delta^2 + \delta^3).
\end{align}

\subsubsection{The parameters of the iterations}

The sequence of frequency-energy levels $\{(\Xi_{(n)},\mcE_{(n)}\}$ are chosen as follows. Recall that $\{(\Xi_{(0)},\mcE_{(0)})\}$ has already been determined above. Set
\begin{align}
\Xi_{(n+1)} = C_{\vartheta} N_{(n+1)} \Xi_{(n)},  \quad \text{ for } n \geq 0, \label{eq:iter-Xi} \\
\mcE_{(n+1),u} = \mcE_{(n),1}, \quad \mcE_{(n+1),1}=2\mcE_{(n),2}, \quad \mcE_{(n+1),2} = 2 \mcE_{(n),3}, \quad \text{ for } n \geq 0, \label{eq:iter}\\
\mcE_{(1),3} = \eps, \quad \mcE_{(n+1),3} = \mcE_{(n),3}^{1+\delta}  \quad \text{ for } n \geq 1, \label{eq:iter-E_3}\\
\ell_{(n)} = \Xi_{(n)}^{-1/2}\mcE_{(n),u}^{-1/4},  \quad \mcD_{(n+1)} = N(\mcD_{(n)};50\ell_{(n)}^{2},50\ell_{(n)})\quad \text{ for } n \geq 0. \label{def:ell_(n)}
\end{align}

Now we choose a sequences of parameters $\{N_{(n)}\}$ to apply Lemma \ref{Lemma:Iteration}. Note that $\{N_{(n)}\}$ and $(\Xi_{(n)},\mcE_{(n)})$ need to satisfy \eqref{eq:ell} and \eqref{ineq:N}.
Furthermore,
\eqref{ineq:N} requires that $\mcE_{(n+1),3} \geq N_{(n+1)}^{-1/3}\mcE_{(n),u}^{1/2} \mcE_{(n),1}^{1/2}$, that is,
\begin{align}
N_{(n+1)} \geq    \left(\frac{\mcE_{(n),u}}{\mcE_{(n),3}}\right)^{3/2}\left(\frac{\mcE_{(n),1}}{\mcE_{(n),3}}\right)^{3/2}, \quad \text{for } n \geq 0. \label{ineq:N_(n+1)-1}
\end{align}
Accordingly,
we set
\begin{align}
N_{(1)} = (\mcE_{(0),u}\mcE_{(0),1})^{3/2}\eps^{-3}, \label{eq:N_(1)}
\end{align}
and, for $n \geq 1$,
\begin{align}
N_{(n+1)} &= \left(\frac{\mcE_{(n),u}}{\mcE_{(n),3}}\right)^{3/2}\left(\frac{\mcE_{(n),1}}{\mcE_{(n),3}}\right)^{3/2}\mcE_{(n),3}^{-3\delta} = (\mcE_{(n),u}\mcE_{(n),1})^{3/2}\mcE_{(n),3}^{-3(1+\delta)}. \label{eq:N_(n+1)-1}
\end{align}

It follows from the definitions \eqref{eq:iter}, \eqref{eq:iter-E_3} and \eqref{eq:N_(n+1)-1} that
\begin{align}
\mcE_{(n),3} &= \eps^{(1+\delta)^{n-1}} \quad \text{for } n \geq 1, \\
(\mcE_{(n),u},&\mcE_{(n),1},\mcE_{(n),2}) = (2^2\eps^{(1+\delta)^{n-4}}, 2^2\eps^{(1+\delta)^{n-3}}, 2 \eps^{(1+\delta)^{n-2}}) \quad \text{for } n \geq 4, \label{eq:E-geq-3}\\
\mcE_{(n),u} &= \mcE_{(n-1),1} = 2 \mcE_{(n-2),2} = 2^2 \mcE_{(n-3),3}, \quad \text{for } n \geq 4, \\
\mcE_{(n),1} &= 2 \mcE_{(n-1),2} = 2^2 \mcE_{(n-2),3} =  2^2 \mcE_{(n-3),3}^{1+\delta}, \quad \text{for } n \geq 4,  \label{eq:E_(1,n)}\\
\mcE_{(n),3} &= \mcE_{(n-1),3}^{1+\delta} = \mcE_{(n-2),3}^{(1+\delta)^2} = \mcE_{(n-3),3}^{(1+\delta)^3} \quad \text{for } n \geq 4,\\
N_{(n+1)} &= (\mcE_{(n),u}\mcE_{(n),1})^{3/2}\mcE_{(n),3}^{-3(1+\delta)} = 2^6 \mcE_{(n-3),3}^{-3(1+\delta)^4+3(2+\delta)/2}\\
&= 2^6  \mcE_{(n-3),3}^{-b} = 2^6 \eps^{-(1+\delta)^{n-4}b} \quad \text{for } n \geq 4. \label{eq:iter-N}
\end{align}

The inequality \eqref{ineq:N_(n+1)-1} follows from the definitions of $N_{(n+1)}$ and the fact that $\mcE_{(n),3} \leq 1$.
It is straightforward to verify that  the inequality \eqref{E_i} holds for $(\Xi_{(0)},\mcE_{(0)})$. Suppose it holds for $(\Xi_{(n)},\mcE_{(n)})$ for some $n \geq 0$. Then
\begin{align*}
4\mcE_{(n+1),3} = 4 \mcE_{(n),3}^{1+\delta} \leq  2\mcE_{(n+1),2} = 4 \mcE_{(n),3}  \leq \mcE_{(n+1),1} = 2 \mcE_{(n),2}  \leq \mcE_{(n+1),u} = \mcE{(n),1}.
\end{align*}
Thus the inequality \eqref{E_i} is verified for all $n \geq 0$.

\subsubsection{The iteration step}
Starting from $(u_{(0)},v_{(0)},R_{(0)})$, suppose that one has obtained functions $(u_{(n)},v_{(n)},R_{(n)})$ which solve the approximate system  \eqref{eq:PrStress} with  frequency-energy levels below $(\Xi_{(n)},\mcE_{(n)})$ and $\text{supp}_{t,y} R_{(n)} \subset \mcD_{(n)}$, iteratively by applying Lemma \ref{Lemma:Iteration} to $(u_{(n-1)},v_{(n-1)},R_{(n-1)})$ with  frequency-energy levels below $(\Xi_{(n-1)},\mcE_{(n-1)})$ and $\text{supp}_{t,y} R_{(n-1)} \subset \mcD_{(n-1)}$.
We first establish some  bounds on the supports of $R_{(n)}$ and $\|v_{(n)}\|_{L^{\infty}}$.

It follows from  \eqref{E_i},  \eqref{eq:iter-Xi},\eqref{eq:iter}, \eqref{def:ell_(n)}, and \eqref{ineq:N_(n+1)-1} that
\begin{align}
\frac{\ell_{(n+1)}}{\ell_{(n)}} &= \left(\frac{\Xi_{(n)}}{\Xi_{(n+1)}}\right)^{1/2} \left(\frac{\mcE_{(n),u}}{\mcE_{(n+1),u}}\right)^{1/4} \leq N_{(n+1)}^{-1/2} \mcE_{(n),u}^{1/4}\mcE_{(n),1}^{-1/4} \leq \frac{\mcE_{(n),3}}{\mcE_{(n),1}}\left(\frac{\mcE_{(n),3}}{\mcE_{(n),u}}\right)^{1/2} \nonumber \\
&  \leq \frac{1}{8} \quad \text{ for }  n \geq 0. \label{ineq:ell_(n+1)-ell_(n)}
\end{align}
Note that \eqref{eq:E_(0)} and \eqref{ineq:Xi_0} imply that
\begin{align*}
100 \ell_{(0)}^2 \leq \rho, \quad 100 \ell_{(0)} \leq \rho^{1/2}.
\end{align*}
It follows from \eqref{def:tildemcD}, \eqref{def:ell_(n)} and \eqref{ineq:ell_(n+1)-ell_(n)} that for $n \geq 0$,
\begin{align}
\mcD_{(n)} \subset N(\mcD_{(0)};50 \sum_{k=1}^n \ell_{(k)}^2, 50 \sum_{k=1}^n \ell_{(k)} ) \subset N(\mcD_{(0)};100 \ell_{(0)}^2, 100 \ell_{(0)} ) \subset \tilde{\mcD}. \label{est-supp-D_n}
\end{align}
Recalling  \eqref{eq:E_(0)} and \eqref{eq:iter}, one has
\begin{align}
\mcE_{(0),1} = 25 \bar{\mcE},\quad  \mcE_{(1),1}=\mcE_{(2),1}=4 \bar{\mcE}, \quad \mcE_{(n),1}=4 \eps^{(1+\delta)^{n-3}}, \text{ for }n \geq 3. \label{eq:mcE_(1)}
\end{align}
It follows from the estimates \eqref{est-w} that
\begin{align}
\|(u_{(n)},v_{(n)})\| &\leq \|(u_{(0)},v_{(0)})\| + \sum_{k=1}^{n-1} \|(u_{(k+1)}-u_{(k)},v_{(k+1)}-v_{(k)})\| \\
&\leq \bar{\mcE} + C_{\vartheta}\sum_{k=1}^{n-1} \mcE_{(k),1}^{1/2} \leq 50 C_{\vartheta} \bar{\mcE}. \label{est-v_n}
\end{align}
Note that \eqref{eq:E_(0)} and \eqref{ineq:Xi_0} yield
\begin{align*}
\Xi_{(0)}\mcE_{(0),u}^{1/2} \geq (100C_{\vartheta} \bar{\mcE})^2.
\end{align*}
Suppose that $\Xi_{(n)}\mcE_{(n),u}^{1/2} \geq (100C_{\vartheta} \bar{\mcE})^2$ for some $n \geq 0$. Then \eqref{ineq:N_(n+1)-1} implies
\begin{align*}
\Xi_{(n+1)} \mcE_{(n+1),u}^{1/2} &= C_{\vartheta} N_{(n+1)}\Xi_{(n)} \mcE_{(n),1}^{1/2} \geq C_{\vartheta}  N_{(n+1)} \frac{\mcE_{(n),1}^{1/2}}{\mcE_{(n),u}^{1/2}} (100C_{\vartheta} \bar{\mcE})^2 \\
&\geq \frac{\mcE_{(n),u}^{3/2}\mcE_{(n),1}^{3/2} \mcE_{(n),1}^{1/2}}{\mcE_{(n),3}^3 \mcE_{(n),u}^{1/2}} (100C_{\vartheta} \bar{\mcE})^2 \geq (100C_{\vartheta} \bar{\mcE})^2  \geq (1+\|v_{(n)}\|_{L^{\infty}} )^{2},
\end{align*}
where the last inequality follows from \eqref{est-v_n}. Thus \eqref{eq:ell}  holds for all $n \geq 0$.

Set
\begin{align}
e_{(n+1)}^{1/2} &= 2 \mcE_{(n),1}^{1/2} \chi_{N(\mcD_{(n)};2\ell_{(n)}^{2},2\ell_{(n)})} * ( \bar{\eta}_{\ell_{(n)}^2}(t)\bar{\eta}_{\ell_{(n)}}(y)),
\end{align}
where $\chi_{N(\mcD_{(n)};2\ell_{(n)}^{2},2\ell_{(n)})}$ denotes the indicator function for the set $N(\mcD_{(n)};2\ell_{(n)}^{2},2\ell_{(n)})$.
Then $\e_{(n+1)}^{1/2}$ is a smooth function with
\begin{align*}
\supp \e_{(n+1)} \subset N(\mcD_{(n)};3\ell_{(n)}^2,3\ell_{(n)}),  ~
\e_{(n+1)}^{1/2} = 2 \mcE_{(n),1}^{1/2} \text{ on }   N(\mcD_{(n)};\ell_{(n)}^2,\ell_{(n)}), \\
\|\e_{(n+1)}^{1/2}\|_{C^0} \leq 2\mcE_{(n),1}^{1/2}, \quad \|\ptl_t^{\alpha} \ptl_y^{\beta} (\e_{(n+1)}^{1/2})\|_{C^0} \leq C_{\alpha,\beta}  \ell_{(n)}^{-(2\alpha + \beta)} \mcE_{(n),1}^{1/2}.
\end{align*}
It follows from \eqref{est-supp-D_n} that
\begin{align}
N(\supp \e_{(n+1)};10\ell_{(n)}^2,10\ell_{(n)}) \subset N(\mcD_{(n)};13\ell_{(n)}^2,13\ell_{(n)}) \subset \mcD_{(n+1)} \subset \tilde{\mcD}. \label{supp-e_(n+1)}
\end{align}
It is straightforward to verify that $\e_{(n+1)}$ satisfies the estimates \eqref{supp-e} and \eqref{est-e}.

In order to apply Lemma \ref{Lemma:Iteration}, it remains to show that
\begin{align}
N_{(n+1)} \geq  \Xi_{(n)}^{\vartheta}, \quad \text{ for } n \geq 0. \label{ineq:N_(n+1)-Xi}
\end{align}
Indeed, it follows from \eqref{eq:N_(1)}, \eqref{eq:N_(n+1)-1} and \eqref{eq:iter-N} that the following rough bounds hold
\begin{align*}
 2^3 \eps^{-3\delta} \leq N_{(k)} \leq 10^3 \bar{\mcE}^3 \overline\eps^{-3(1+\delta)^3} \quad \text{ for } k=1,\cdots,5.
\end{align*}
Thus using \eqref{ineq:eps-1}, one can obtain
\begin{align*}
\Xi_{(k)} &= C_{\vartheta}^k N_{(k)}N_{(k-1)} \cdots N_{(1)} \Xi_{(0)}  \leq C_{\vartheta}^k (10 \bar{\mcE} \overline\eps^{-(1+\delta)^3})^{3k} \Xi_{(0)} \leq  C_{\vartheta}^4 (10 \bar{\mcE} \overline\eps^{-8})^{12} \Xi_{(0)}\\
& \leq 10^{12} \eps^{-120}, \quad \text{ for } k=0,\cdots,4.
\end{align*}
This, together with \eqref{def:vartheta}, shows  that
\[ N_{(k)} \geq 2^3 \eps^{-3\delta} \geq (10^{12} \eps^{-120})^{\vartheta} \geq  \Xi_{(k-1)}^{\vartheta}, \quad \text{ for } k=1,\cdots,5.\]
Now we use induction on $n$. Suppose that for some $n \geq 5$,
\begin{align}
N_{(n)} \geq  \Xi_{(n-1)}^{\vartheta}. \label{N_n-Xi}
\end{align}

Then using the expression \eqref{eq:iter-N} for $N_{(n+1)}$ leads to
\begin{align*}
N_{(n+1)} &= 2^6 \eps^{-(1+\delta)^{n-4}b} \geq 2^6 (2^6C_{\vartheta})^{\vartheta}  \eps^{-(1+\delta)^{n-5}b(1+\vartheta)} =  C_{\vartheta}^{\vartheta} N_{(n)}^{1+\vartheta} \\
&\geq  (C_{\vartheta} N_{(n)} \Xi_{(n-1)})^{\vartheta} = \Xi_{(n)}^{\vartheta},
\end{align*}
where the first inequality follows from  \eqref{def:vartheta} and \eqref{ineq:eps-1}, and the induction assumption \eqref{N_n-Xi} is used in the last inequality. This confirms \eqref{ineq:N_(n+1)-Xi}.

Now one can apply Lemma \ref{Lemma:Iteration} to obtain $(u_{(n+1)},v_{(n+1)},R_{(n+1)})$ with  frequency-energy levels below $(\Xi_{(n)},\mcE_{(n)})$. It follows from \eqref{supp-w-tildeR} and \eqref{supp-e_(n+1)} that
\begin{align}
\supp_{t,y}(u_{(n+1)}-u_{(n)},v_{(n+1)}-v{(n)},R_{(n+1)}) \subset  \mcD_{(n+1)}. \label{supp-w_(n+1)}
\end{align}

\subsubsection{Convergence and regularity}
 Denote $w_{(n+1)} = (u_{(n+1)}-u_{(n)},v_{(n+1)}-v_{(n)})$.
Then \eqref{est-w} implies that for $0 \leq |\alpha|+ 2\beta \leq 1$,
\begin{align}
\|\nabla_x^{\alpha}\ptl_y^{\beta}w_{(n+1)}\|_{C^0} &\leq C_{\vartheta} (N_{(n+1)}\Xi_{(n)})^{|\alpha|} (N_{(n+1)}^{1/2} \Xi_{(n)}^{1/2} \mcE_{(n),1}^{1/4})^{\beta} \mcE_{(n),1}^{1/2}, \label{ineq:w_(n+1)}\\
\|\ptl_t w_{(n+1)}\|_{C^0} &\leq \| (\ptl_t + u_{(n)} \cdot \nabla_x + v_{(n)}\ptl_y) w_{(n+1)} \| + \|u_{(n)} \cdot \nabla_x w_{(n+1)}\| + \|v_{(n)} \ptl_y w_{(n+1)}\| \nonumber \\
& \leq  C_{\vartheta} N_{n+1}\Xi_{n}\mcE_{(n),1}^{1/2}.  \label{ineq:w_(n+1)-t}
\end{align}

Note that $\{\mcE_{(n),1}^{1/2}\}$ is a Cauchy series due to \eqref{eq:E-geq-3}. So it follows from \eqref{ineq:w_(n+1)} and \eqref{ineq:w_(n+1)-t}
that the sequence $\{(u_{(n)},v_{(n)})\}$ converges uniformly to a continuous function $(u,v)$. Since $(u_{(n)},v_{(n)},R_{(n)})$ solves the approximate system  \eqref{eq:PrStress} with  frequency-energy levels below $(\Xi_{(n)},\mcE_{(n),u},\mcE_{(n),1},\mcE_{(n),2},\mcE_{(n),3})$,
\begin{align*}
\|R_{(n)}\|_{C^0}  \leq \mcE_{(n),1}+\mcE_{(n),2}+\mcE_{(n),3} \to 0,
\end{align*}
due to \eqref{eq:E-geq-3}, thus $(u,v)$ is a weak solution to the Prandtl system \eqref{eq:Prandtl}.

Next we consider the regularity of the solutions.
It follows from the estimates \eqref{ineq:w_(n+1)}, \eqref{ineq:w_(n+1)-t} and standard interpolations that for $\alpha \in (0,1)$,
\begin{align*}
\| w_{(n+1)} \|_{C^{\alpha}_{t,x}} \leq C_{\alpha} (N_{n+1}\Xi_{n})^{\alpha} \mcE_{(n),1}^{1/2}.
\end{align*}
Set
\begin{align*}
a_{n} =  (N_{(n)}\Xi_{(n-1)})^{\alpha} \mcE_{(n-1),1}^{1/2}.
\end{align*}
It follows from  \eqref{eq:iter-Xi}, \eqref{eq:iter}, \eqref{eq:E_(1,n)} and \eqref{eq:iter-N} that for $n \geq 4$,
\begin{align*}
\frac{a_{n+1}}{a_{n}} &= \frac{(N_{(n+1)}\Xi_{(n)})^{\alpha} \mcE_{(n),1}^{1/2}}{(N_{(n)}\Xi_{(n-1)})^{\alpha} \mcE_{(n-1),1}^{1/2}} = C_{\vartheta}^{\alpha}N_{(n+1)}^{\alpha} \left(\frac{\mcE_{(n),1}}{ \mcE_{(n-1),1}}\right)^{1/2}= C_{\vartheta}^{\alpha}N_{(n+1)}^{\alpha} \left(\mcE_{(n-3),3}^{\delta}\right)^{1/2}\\
&= C_{\vartheta}^{\alpha}2^{6\alpha}\mcE_{(n-3),3}^{\delta/2 - b\alpha} = C_{\vartheta}^{\alpha}2^{6\alpha} \eps^{(\delta/2 - b\alpha)(1+\delta)^{n-4}} .
\end{align*}
Set $\alpha = \alpha_0$. Then \eqref{ineq:alpha} shows that
\begin{align*}
\gamma:=\frac{\delta}{2} - b \alpha_0 = \frac{\delta}{2} - 3\delta(\frac{7}{2} + 6\delta + 4\delta^2 + \delta^3)\alpha_0 >0.
\end{align*}
Hence $\limsup_n \frac{a_{n+1}}{a_{n}} = 0$ and
thus $\{\| w_{(n+1)} \|_{C^{\alpha_0}_{t,x}}\}$ is a Cauchy series.
Similarly, using the estimates \eqref{ineq:w_(n+1)}, \eqref{ineq:w_(n+1)-t} and standard interpolations one can get that,  for $\beta\in (0,1)$,
\begin{align*}
\| w_{(n+1)} \|_{C^{\beta}_{y}} \leq C_{\beta} (N_{(n+1)}^{1/2} \Xi_{(n)}^{1/2} \mcE_{(n),1}^{1/4})^{\beta}  \mcE_{(n),1}^{1/2}.
\end{align*}
Set
\begin{align*}
b_{n} =  (N_{(n)}^{1/2} \Xi_{(n-1)}^{1/2} \mcE_{(n-1),1}^{1/4})^{\beta}  \mcE_{(n-1),1}^{1/2}.
\end{align*}
Then, \eqref{eq:iter-Xi}, \eqref{eq:iter}, \eqref{eq:E_(1,n)} and \eqref{eq:iter-N} yield that for $n \geq 4$,
\begin{align*}
\frac{b_{n+1}}{b_{n}} &= \frac{(N_{(n+1)}^{1/2} \Xi_{(n)}^{1/2} \mcE_{(n),1}^{1/4})^{\beta}  \mcE_{(n),1}^{1/2}}{(N_{(n)}^{1/2} \Xi_{(n-1)}^{1/2} \mcE_{(n-1),1}^{1/4})^{\beta}  \mcE_{(n-1),1}^{1/2}} = C_{\vartheta}^{\beta/2}N_{(n+1)}^{\beta/2} \left(\frac{\mcE_{(n),1}}{ \mcE_{(n-1),1}}\right)^{\frac{1}{2}+\frac{\beta}{4}}\\
&= C_{\vartheta}^{\beta/2}N_{(n+1)}^{\beta/2} \left(\mcE_{(n-3),3}^{\delta}\right)^{\frac{1}{2}+\frac{\beta}{4}}
= C_{\vartheta}^{\beta/2}2^{3\beta}\mcE_{(n-3),3}^{(\delta + \frac{\delta \beta}{2} - b \beta)/2} = C_{\vartheta}^{\beta/2} \eps^{(\delta + \frac{\delta \beta}{2} - b \beta)(1+\delta)^{n-4}/2} .
\end{align*}
Set $\beta = \beta_0$. It follows from \eqref{ineq:alpha} that
\begin{align*}
\gamma:=\delta + \frac{\delta \beta_0 }{2}- b \beta_0 = \delta(1+\frac{\beta_0}{2}) - 3\delta(\frac{7}{2} + 6\delta + 4\delta^2 + \delta^3)\beta_0 >0.
\end{align*}
Hence $\limsup_n \frac{b_{n+1}}{b_{n}} = 0$
and  thus $\{\| w_{(n+1)} \|_{C^{\beta_0}_{y}}\}$ is a Cauchy series.
We thus have proved the estimates \eqref{ineq:u-u_S}.

Furthremore, \eqref{est-supp-D_n} and \eqref{supp-w_(n+1)} yield that
\begin{align*}
\text{supp}_{t,y} (u - \underline{u},v - \underline{v}) \subset \bigcup_{n=1}^{\infty} \supp_{t,y} w_k \subset \bigcup_{n=1}^{\infty} \mcD_{(n)} \subset  N(\text{supp}_{t,y}(\underline{u} - u_C,\underline{v}-v_C);\rho,\rho^{1/2}).
\end{align*}

\subsubsection{Weak convergence of the solution sequence}
The above scheme shows that there exists a  constant $\bar{\eps}>0$, such that for any choice of $\eps < \bar{\eps}$ in \eqref{ineq:eps-1}, there exists a weak solution to the system \eqref{eq:Prandtl} of the form
\begin{align*}
(u,v) = (\underline{u},\underline{v}) + \sum_{n=1}^{\infty}w_{(n)}
\end{align*}
satisfying the estimate \eqref{ineq:u-u_S}. It follows from \eqref{eq:mcE_(1)} and \eqref{ineq:w_(n+1)} that
\begin{align*}
\|(u - \underline{u}, v - \underline{v})\|_{C^0} & \leq \sum_{n=1}^{\infty}\|w_{(n)}\|_{C^0} \leq C \sum_{n=0}^{\infty} \mcE_{(n),1}^{1/2} \leq  C \bar{\mcE} + C \sum_{k=0}^{\infty} \eps^{(1+\delta)^{k}/2} \\
& \leq C \bar{\mcE} +  2C \eps^{1/2} + \leq 2C \bar{\mcE}.
\end{align*}
In particular one has
\begin{align*}
\sum_{n=4}^{\infty}\|w_{(n)}\|_{C^0}
 \leq 2C \eps^{1/2}.
\end{align*}
For $n=1,2,3$, let $\varphi \in C_c^{\infty}(\Rp \times \T^2 \times \Rp)$ be a smooth test function. Recall that
\begin{gather*}
w_{(n)} = \sum_I e^{i\lambda_{(n)} \xi_{(n),I} } \tilde{W}_{(n),I},\\
\lambda_{(n)} = B^3 N_{(n)} \Xi_{(n-1)}, \quad e^{i\lambda_{(n)} \xi_{(n),I} } =   \frac{\nabla_x \xi_{(n),I} }{i \lambda_{(n)} |\nabla_x \xi_{(n),I}|^2}  \nabla_x(e^{i \lambda_{(n)} \xi_{(n),I}}).
\end{gather*}
Using integration by parts,  \eqref{est-tilde-W}, \eqref{ineq:eps-1}, \eqref{eq:N_(1)}, \eqref{eq:N_(n+1)-1}, and \eqref{eq:mcE_(1)}, one gets
\begin{align*}
&\left|\int w_{(n)}\varphi \,dxdydt\right| \leq  \sum_I \left|\int_{Q_I} e^{i\lambda_{(n)} \xi_{(n),I}}\tilde{W}_{(n),I}\varphi \,dxdydt\right| \\
&\leq C  \lambda_{(n)}^{-1}\sum_I  \|\nabla_x(\tilde{W}_{(n),I} \varphi)\|_{C^0}  |\supp \varphi \cap Q_I| \leq C\|\varphi\|_{C^1}  |\supp \varphi| N_{(n)}^{-2/3} \mcE_{(n-1),1}^{1/2}\\
&\leq  C \|\varphi\|_{C^1}  |\supp \varphi|\eps^{1/2}, \quad \text{ for } n=1,2,3.
\end{align*}
Hence for any $\varphi \in C_c^{\infty}(\Rp \times \T^2 \times \Rp)$ one has
\begin{align*}
\left|\int (u - \underline{u}, v - \underline{v})\varphi dx dy dt\right| &\leq \sum_{n=1}^{\infty}\left|\int w_{(n)}\varphi dx dy dt\right| \leq  C \|\varphi\|_{C^1}  |\supp \varphi| \eps^{1/2}.
\end{align*}
Let $(u_k,v_k)$ be the weak solutions corresponding to a sequence of positive numbers $\{\eps_k\}$ with $\eps_k < \bar{\eps}, \eps_k \to 0$.
It follows from the above estimates   and the standard density argument that $(u_k,v_k) \to (\underline{u},\underline{v})$ in the weak-$*$ topology on $L^\infty(\Rp \times \T^2 \times \Rp)$. This finishes the proof of Theorem \ref{Thm:Prandtl}.

\subsection{Proof of Corollary \ref{Corollary:ShearFlow}}
Let $\phi_{\varepsilon}(t,x,y)$ be a smooth bump function supported in a small ball $B_{\varepsilon}(t_0,x_0,y_0) \subset B_{2\varepsilon}(t_0,x_0,y_0) \subset  \Rp \times \T^2 \times \Rp$, such that $\ptl_{yy}^2 \ptl_{x_2} \phi_{\varepsilon}$ changes signs. Set $(\underline{u},\underline{v}) = (u_S,0) + (\ptl_{y} \ptl_{x_2} \phi_{\varepsilon}, 0, -\ptl_{x_1} \ptl_{x_2} \phi_{\varepsilon})$. It is clear that $(\underline{u},\underline{v})$ satisfies the conditions \eqref{eq:div-bu-bv} and \eqref{u-u_S}.
Applying  Theorem \ref{Thm:Prandtl} to $(\underline{u},\underline{v})$ with $\rho = \varepsilon$, one can obtain a sequence of H\"{o}lder continuous weak solutions $\{(u_k,v_k)\}_{k=1}^{\infty}$ to the system \eqref{eq:Prandtl} satisfying the estimates \eqref{ineq:u-u_S}, such that
\begin{align*}
\supp_{t,y} (u_k-u_S,v_k) &\subset \supp_{t,y} (u_k-\underline{u},v_k-\underline{v}) \cup \supp_{t,y} (\ptl_{y} \phi_{\varepsilon}, 0, -\ptl_{x_1} \phi_{\varepsilon})\\
& \subset N(\text{supp}_{t,y}(\phi_{\varepsilon});\varepsilon,\varepsilon^{1/2}) \subset B_{2\varepsilon}(t_0,y_0) \subset \Rp  \times \Rp,
\end{align*}
and $(u_k,v_k) \rightharpoonup (\underline{u},\underline{v})$ in the weak-$*$ topology on $L^\infty(\Rp \times \T^2 \times \Rp)$.
Furthermore,
\begin{align*}
\ptl_y u_k  \rightharpoonup \ptl_y u_S + (\ptl_{yy}^2 \ptl_{x_2} \phi_{\varepsilon},0), \quad v_k \rightharpoonup - \ptl_{x_1} \ptl_{x_2} \phi_{\varepsilon},
\end{align*}
in the sense of distribution.
Thus for $k$ sufficiently large and $\eps$ sufficiently small, $u_k$ is not monotonic in $y$ and $(u_k,v_k)$ has motion transverse to the outflow $(U,0)$.

\subsection{Proof of Theorem \ref{Thm:Euler-partial}}\label{sec:Euler-partial}
The proof of Theorem \ref{Thm:Euler-partial} is just a slight modification of the proof of Theorem \ref{Thm:Prandtl}. We outline the main differences in the constructions. Similar to \eqref{eq:PrStress} we consider the following approximate system
\begin{align}\label{eq:EPStress}
\begin{cases}
\ptl_t u_{(n)} + \nabla \cdot (u_{(n)} \otimes u_{(n)} )  - \ptl_{yy}^2 u_{(n)} + \nabla P_{(n)} = \nabla \cdot R_{(n)},\\
\nabla \cdot u_{(n)} = 0,
\end{cases}
\end{align}
where the stress $R = \begin{pmatrix} S & Y\\ Y^t & r \end{pmatrix}$ is a $3 \times 3$ symmetric matrix.
The main difference is that we have to add an extra correction to eliminate the $r$ component of the stress.  Definition \ref{Def:FE-Levels} of frequency-energy level is unchanged except replacing $(u,v)$ by $u$, $(S,Y)$ by $(S,Y,r)$.
Then we have the following variant of Lemma \ref{Lemma:Iteration}.
\begin{Lemma}\label{Lemma:Iteration'}
	Given $\vartheta> 0$, there exists a constant $C_\vartheta$ such that the following holds. Suppose that $(u,R)=(u,S,Y,r)$ is a smooth solution  to \eqref{eq:EPStress} with frequency-energy levels below $(\Xi,\mcE)$, with $\ell := \Xi^{-1/2}\mcE_u^{-1/4} \leq 1$.
	Let $\e(t)$  be a given non-negative function satisfying
	\begin{align}
	\e(t) \geq 4 \mcE_1 \text{ on }N(\text{supp}_{t} R;\ell^2), \quad N(\supp e, 50\ell^2) \subset \Rp, \label{supp-e'}\\
	\| (\frac{d}{dt})^{\alpha} (\e^{1/2})\|_{L^{\infty}} \leq C_{\alpha}  \ell^{-2\alpha}\mcE_1^{1/2},  \quad 0 \leq \alpha \leq 1.
	\label{est-e'}
	\end{align}
	Then for any $N > 0$  satisfying \eqref{ineq:N}，
	there exists a smooth solution $(\tilde{u},\tilde{R})$ to the system \eqref{eq:EPStress} with frequency-energy levels below  $(\tilde{\Xi},\tilde{\mcE})=(\tilde{\Xi},\tilde{\mcE_u},\tilde{\mcE_1},\tilde{\mcE_2},\tilde{\mcE_3})$ as given in \eqref{tilde-FE-levels}.
	Furthermore, the correction $w = \tilde{u} - u$ satisfies the estimates \eqref{est-w}, and
	\begin{align}
	\text{supp}_{t} (w,\tilde{R}) \subset  N(\supp \e;\ell^2). \label{supp-w-tildeR'}
	\end{align}
\end{Lemma}
\begin{proof}[Proof of Lemma \ref{Lemma:Iteration'}]
One can follow closely the proof of  Lemma \ref{Lemma:Iteration} with a few modifications.
Let $H, H' \in \mathbb{Z}$ satisfy
\begin{align*}
2^{-H}  \leq \ell_x < 2^{-(H-1)}, \quad 2^{-H'}  \leq \ell_y < 2^{-(H'-1)}.
\end{align*}
For $\tilde{\kappa} = (\kappa_1,\kappa_2, \kappa_3) \in  (\mathbb{Z}/2^{H}\mathbb{Z})^2 \times (\mathbb{Z}/2^{H'}\mathbb{Z})$, set
\[ \tilde{\psi}_{\tilde{\kappa}}(x_1,x_2,y) = \eta_H\left(2^H(x_1-\kappa_1 2^{-H})\right)\eta_H\left(2^H(x_2-\kappa_2 2^{-H})\right)\eta_{H'}\left(2^{H'}(y-\kappa_3 2^{-H'})\right),\]
and for $\kappa = (\kappa_0,\tilde{\kappa}) \in \mathbb{Z} \times  (\mathbb{Z}/2^{H}\mathbb{Z})^2 \times (\mathbb{Z}/2^{H'}\mathbb{Z})$,  let  $\psi_{\kappa}$ be defined by
\begin{align}
(\ptl_t + u_{\ell} \cdot \nabla) \psi_{\kappa} = 0, \label{eq:psi_kappa'} \quad
\psi_{\kappa}(\kappa_0\tau,\cdot) = \tilde{\psi}_{\tilde{\kappa}}(\cdot),
\end{align}
where $u_{\ell} = u * \tilde{\eta}_{\ell} $.
The constructions of localized waves are modified as follows.
For each index $I=(\kappa(I),s(I)) \in \mathbb{Z} \times  (\mathbb{Z}/2^{H}\mathbb{Z})^2 \times (\mathbb{Z}/2^{H'}\mathbb{Z}) \times \{ +,- \} $, set
\begin{equation*} W_I = (U_I,V_I) = \eta_I \psi_I A_I = \eta_I \psi_I (a_I\f_1,b_I), \quad G_I = \eta_I \psi_I  (0,0,g_I) = \tilde{g}_I \vec{e}_3,
\end{equation*}
where $a_I, b_I$ are defined in \eqref{def:a}, $g_I = \sqrt{\frac{(e + r)(q_I)}{2}-b_I^2}$ and $\vec{e}_3 = (0,0,1)$. Similar to the estimates \eqref{est-omega}, by \eqref{supp-e'}, one can verify that $g_I$ is well-defined with
\begin{align}
g_I \leq C \mcE_1^{1/2}. \label{est-omega'}
\end{align}
Define the correction $w = \tilde{u} - u$ as
\begin{align*}
w =  \sum_I (e^{ i \lambda \xi_I}\tilde{W}_I + e^{i\lambda 100 \xi_I}\tilde{G}_I) = \sum_I (e^{ i \lambda \xi_I}(W_I +
\delta W_I) + e^{i\lambda 100 \xi_I}(G_I +
\delta G_I)),
\end{align*}
where $\xi_I=\xi_I(x,t)$ is defined in \eqref{def:xi_I} (with $u_{\kappa(I)}$ replaced by $(u^1,u^2)(q_{\kappa(I)})$), $\delta W_I$ and $\delta G_I$ are small corrections to ensure that $\nabla \cdot w = 0$. The explicit expressions are given in \eqref{def:w_I}, (with $\frac{1}{\lambda^2|\nabla \xi_I|^2}e^{i\lambda\xi_I}W_I$ replaced by $\frac{1}{\lambda^2|\nabla \xi_I|^2}e^{i\lambda\xi_I}W_I + \frac{1}{\lambda^2|100\nabla \xi_I|^2}e^{i\lambda 100 \xi_I}G_I$). Then it can be checked that $w$ and $\tilde{u}$ satisfy the same estimates as proved in Section \ref{subsec:est-velocity}.

Plugging in the correction $\tilde{u} = u+w$ and using the expressions \eqref{eq:mbR} as before, one can decompose the new stress $\tilde{R}$ as:
\begin{align}
\tilde{R}
&= -\begin{pmatrix}
\sum_{i\neq 1}S_{i}\f_{i}\otimes \f_{i} & \sum_{i\neq 1}Y_{i}\f_{i}\\
\sum_{i\neq 1}Y_{i}\f_{i}^t & 0
\end{pmatrix} + (R_S + R_{M}  + R_H + R_T + R_{L})
\end{align}
such that (denoting $\tilde{W}_I = (\tilde{U}_I, \tilde{V}_I)$)
\begin{align*}
\nabla \cdot R_S =&   \nabla \cdot \begin{pmatrix}
-(e_{\ell}+S_{\ell}) \f_1 \otimes \f_1 + \sum_I  \tilde{U}_I \otimes \tilde{U}_{\bar{I}} &   -Y_{\ell} \f_1 + \sum_I \tilde{V}_I \tilde{U}_{\bar{I}}\\
(-Y_{\ell} \f_1 + \sum_I \tilde{V}_I \tilde{U}_{\bar{I}})^t & -(e_{\ell}+r_{\ell}) + \sum_I  \tilde{g}_I \tilde{g}_{\bar{I}} + \tilde{V}_I \tilde{V}_{\bar{I}}
\end{pmatrix},   \\
\nabla \cdot R_{M}=& \nabla \cdot  \begin{pmatrix}
(e_{\ell}-e(t) + S_\ell-S_1) \f_1 \otimes \f_1  & (Y_\ell-Y_1) \f_1\\
(Y_\ell-Y_1) \f_1^t & e_{\ell}-e(t) + r_\ell-r
\end{pmatrix}  \nonumber \\
&  \quad + \nabla \cdot  ((u-u_{\ell}) \otimes w + w \otimes (u-u_{\ell})),  \\
\nabla \cdot R_H =  & \sum_{J \neq \bar{I}}\nabla \cdot(e^{i \lambda (\xi_I+\xi_J)} \tilde{W}_I \otimes \tilde{W}_J  + e^{i100\lambda(\xi_I + \xi_J)}\tilde{G}_I \otimes \tilde{G}_J) \\
&  \quad + \sum_{I,J}\nabla \cdot (e^{i\lambda(\xi_I + 100\xi_J)}\tilde{W}_I \otimes \tilde{G}_J + e^{i\lambda(\xi_J + 100\xi_I)}\tilde{G}_I \otimes \tilde{W}_J) ,    \\ \nabla \cdot R_T=& (\ptl_t + u_{\ell} \cdot \nabla  - \ptl_{yy}^2) w, \quad \nabla \cdot R_{L}=  \nabla \cdot ( w \otimes u_{\ell}),
\end{align*}
where, similar to \eqref{def:R_ell}, the mollification of $r$ is defined as
\begin{align}
r_{\ell}(t,z)&= \sum_{\kappa} \eta_{\kappa}^2\psi_{\kappa}^2(t,z)r(q_{\kappa}).
\end{align}

The terms in $R_M$ are treated exactly as in Section \ref{sec:moll-error}. For $R_S$, note that
\begin{align*}
-(e_{\ell}+r_{\ell}) + \sum_I  \tilde{g}_I \tilde{g}_{\bar{I}} + \tilde{V}_I \tilde{V}_{\bar{I}} = &= \sum_{\kappa}\eta_{\kappa_0}^2\psi_{\kappa}^2(-(\e+r)(q_{\kappa}) + 2g_{\kappa}^2 + 2b_{\kappa}^2) = 0.
\end{align*}
Hence one can show that the main part of $R_S$ vanishes as in Section \ref{sec:stress-term}.

Analogous to Lemma \ref{Lemma:SolveSymmDiv}, for any given smooth vectors $H = e^{i \lambda \xi}(h^1,h^2,h^3)$ supported in $\hat{Q}=\hat{Q}_{u_\ell}(\tau, 3 \ell_x, 3 \ell_x, 3 \ell_y;q_{\kappa})$, satisfying the estimates \eqref{est-xi}, \eqref{est-h_I-y} and the compatibility conditions
\begin{align}
\int H^l dxdy = 0, \quad  \int z^j H^l - z^l H^jdxdy = 0, \quad j,l=1,2,3, \label{H^l-conditions-3'}
\end{align}
there exists a symmetric $3 \times 3$ matrix $T=T^{jl} \in C_c^{\infty}(\hat{Q})$ solving $\nabla \cdot T = H$, with the estimates \eqref{est-R-div}.
The terms $R_T$ and $R_L$ can be handled 	exactly as before.  From the definitions \eqref{def:xi_I} and \eqref{def:[kappa]}, it holds that $|\nabla(\xi_I + 100\xi_J)| \geq 1$. In view of  the orthogonality conditions
\begin{align*}
\nabla \xi_J \cdot \vec{e}_3 = 0, \quad \nabla \xi_J \cdot A_I = 0,
\end{align*}
one can treat the interaction terms $R_H$ as in Section \ref{sec:R_H}.
\end{proof}
Employing the scheme in Section \ref{subserc:proofPrandtl} with Lemma \ref{Lemma:Iteration'}, starting with the trivial solution $u = 0$,  one can obtain a non-trivial H\"{o}lder continuous weak solutions $u$ to \eqref{eq:Euler-partial}  supported in a compact time interval with estimates \eqref{ineq:u}.  This proves Theorem \ref{Thm:Euler-partial}.

\appendix

\section{Transport estimates} \label{sec:transport-estimates}

\begin{proof}[Proof of Lemma \ref{Lemma:transport-scaling}]
	First we show that for $\ell_i = \delta_U = \delta_f = 1$, the following estimates hold
	\begin{align}
	\sup_{1 \leq |\alpha| \leq s}\|\nabla_{z}^{\alpha} f(t,\cdot)\|_{L^\infty}   \leq \tilde{C}_{s}, \quad \text{ for } t \in [-1,1], s=0,1,\cdots,m, \label{ineq:ptl_f}
	\end{align}
	where  $\tilde{C}_s = C(C_{\alpha},0 \leq |\alpha| \leq s)$ denote generic functions of $C_{\alpha}$.
	By the time reversal symmetry $t \to -t$, it suffices  to show \eqref{ineq:ptl_f} for $t \in [0,1]$.
	
	For $s=0$, since $f$ is constant on the integral curves of $\ptl_t + \bar{U} \cdot \nabla_z$, for any $t$,
	\begin{align*}
	\|f(t,\cdot)\|_{L^{\infty}} = \|f_0\|_{L^{\infty}}.
	\end{align*}
	Suppose that the estimates \eqref{ineq:ptl_f} for $m = s-1$ have been obtained. For any multi-index $|\alpha|=s$, applying $\ptl_z^{\alpha}$ to the equation \eqref{eq:transport-f} yields
	\begin{align*}
	(\ptl_t + \bar{U} \cdot \nabla_z) \ptl_z^{\alpha} f &= -\left(\ptl_z^{\alpha}(\bar{U}^i \ptl_{z_i}f) - \bar{U}^i \ptl_{z_i}\ptl_z^{\alpha}f \right) = -\sum_{\mathclap{\alpha_1+\alpha_2=\alpha,|\alpha_2|\leq s-1}}C_{\alpha_1,\alpha_2}(\ptl_z^{\alpha_1}\bar{U}^i) (\ptl_z^{\alpha_2} \ptl_{z_i}f).
	\end{align*}
	Multiplying by $\ptl_z^{\alpha} f$ and summing for all $|\alpha|=s$, one gets
	\begin{align*}
	\frac{1}{2}\frac{D}{Dt}|\nabla_z^{s} f|^2 = -\sum_{|\alpha_1+\alpha_2|=s} & \Big\{\sum_{|\alpha_2|=s-1}C_{\alpha_1,\alpha_2}(\ptl_z^{\alpha_1+\alpha_2}f)(\ptl_z^{\alpha_2} \ptl_{z_i}f)(\ptl_z^{\alpha_1}\bar{U}^i) \\
	&+ \sum_{|\alpha_2|\leq s-2}C_{\alpha_1,\alpha_2}(\ptl_z^{\alpha_1+\alpha_2}f)(\ptl_z^{\alpha_2} \ptl_{z_i}f)(\ptl_z^{\alpha_1}\bar{U}^i) \Big\},
	\end{align*}
	where $\frac{D}{Dt} = \ptl_t + \bar{U} \cdot \nabla_z$. Hence
	\begin{align*}
	\frac{1}{2}\frac{D}{Dt}|\nabla_z^{s} f|^2 & \lesssim  |\nabla_z \bar{U}||\nabla_z^{s} f|^2+ \sum_{2 \leq j \leq s}|\nabla_z^{j}\bar{U}||\nabla_z^{s-j} f||\nabla_z^{s} f|\\
	& \leq C\left((\sup_{|\alpha| = 1}C_{\alpha}) |\nabla_z^{s} f|^2 + (\sup_{ 2 \leq |\alpha| \leq s}  C_{\alpha}) \tilde{C}_{s-2}   |\nabla_z^{s} f| \right).
	\end{align*}
	The estimates \eqref{ineq:ptl_f} follow from the Gronwall's inequality.
	
	Now let $\ell_1,\cdots,\ell_d, \delta_U, \delta_f$ be positive numbers. Set
	\begin{equation}
	\tau_0 = \delta_{U}^{-1} \min_i \ell_i, \quad (t',z_1',\cdots,z_d')= (\frac{t}{\tau_0},\frac{z_1}{\ell_1},\cdots,\frac{z_d}{\ell_d}).
	\end{equation}
	Then the function $h(t',z') = \delta_f^{-1} f(\tau_0 t',\ell_1 z_1',\cdots,\ell_d z_d')$ solves
	\begin{align*}
	(\ptl_{t'} + \bar{U}' \cdot \nabla_{z'})h = 0, \quad h|_{t'=0}=h_0(z') = \delta_f^{-1}f_0(\ell_1 z_1',\cdots,\ell_d z_d'),
	\end{align*}
	with
	\[ \bar{U}'=(\bar{U}_1',\cdots,\bar{U}_d')(t',z') = (\frac{\tau_0}{\ell_1}\bar{U}_1,\cdots,\frac{\tau_0}{\ell_d}\bar{U}_d)(\tau_0 t',\ell_1 z_1',\cdots,\ell_d z_d'). \]
	It follows from \eqref{ineq:U-f_0} that
	\begin{align*}
	\|\ptl_{z'}^{\alpha} h_0\|_{L^\infty}   &\leq C_{\alpha}, \quad \text{ for } 1 \leq |\alpha| \leq m,\\
	\|\ptl_{z'}^{\alpha} \bar{U}'\|_{L^\infty}&\leq  C_{\alpha}, \quad \text{ for } 0 \leq |\alpha| \leq m.
	\end{align*}
	It follows from \eqref{ineq:ptl_f} that, for $0 \leq |\alpha| \leq m$,
	\begin{align}
	\|\ptl_{z'}^{\alpha} h(t',\cdot)\|_{L^\infty}   \leq \tilde{C}_{\alpha}, \quad \text{ for } |t'| \leq 1.
	\end{align}
	This implies the estimates \eqref{ineq:ptl_z-f} for $f$.
\end{proof}


\begin{proof}[Proof of Lemma \ref{Lemma:flow-scaling}]
	Set
	\begin{equation}
	\begin{aligned}
	(z_1',\cdots,z_d')&= (\frac{z_1}{\ell_1},\cdots,\frac{z_d}{\ell_d}), \\ (\bar{U}_1',\cdots,\bar{U}_d')(t,z')&= (\frac{\bar{U}_1}{\ell_1},\cdots,\frac{\bar{U}_d}{\ell_d})(t,\ell_1 z_1',\cdots,\ell_d z_d'), \label{eq:z-U-scaling}
	\end{aligned}
	\end{equation}
	and denote by $\varphi$ the map $\varphi:(t,z) \to (t,z')$. Then $\varPhi_s:=\varphi \circ \Phi_s \circ \varphi^{-1}$ is the flow generated by
	$(\ptl_t + \bar{U}' \cdot\nabla_{z'})$ in the $(t,z')$ coordinates. Note that
	\begin{align*}
	\left| \frac{d}{ds}\sum_{i=1}^{d}|\varPhi_s^i(t,p') - \varPhi_s^i(t,q')|^2 \right| &= 2 \left| \sum_{i=1}^{d} \left(\bar{U}'_i(\varPhi_s(t,p'))-\bar{U}'_i(\varPhi_s(t,q'))\right) (\varPhi_s^i(t,p') - \varPhi_s^i(t,q')) \right| \\
	&\leq 2 \| \nabla_{z'} \bar{U}' \|_{L^\infty}   |\varPhi_s(t,p') - \varPhi_s(t,q')|^2.
	\end{align*}
	It follows from the Gronwall's inequality that
	\begin{align}
	|\varPhi_s(t,p') - \varPhi_s(t,q')| \leq e^{s\| \nabla_{z'} \bar{U}' \|_{L^\infty} }|p'-q'|. \label{ineq:tlPhi_s'}
	\end{align}
	Due to \eqref{ineq:U-p-q} and \eqref{eq:z-U-scaling}, it holds that
	\begin{align*}
	\|\nabla_{z'} \bar{U}'\|_{L^\infty}  \leq  A_1  (\min_i \ell_i)^{-1} \delta_{U}, \quad |p_i'-q_i'| \leq A_2.
	\end{align*}
	Taking $|s|  \leq \delta_{\bar{U}}^{-1} \min_i \ell_i$ in \eqref{ineq:tlPhi_s'}  yields
	\begin{align*}
	\sum \frac{(\Phi_s^i(t,p) - \Phi_s^i(t,q))^2}{\ell_i^2} =|\varPhi_s(t,p') - \varPhi_s(t,q')|^2 \leq (A_2 e^{A_1})^2.
	\end{align*}
	This shows the desired estimates \eqref{ineq:Phi_s}.
\end{proof}

\section{Proof of Lemma \ref{Lemma:OperatorSymmDiv}} \label{sec:pf-op-symm}

Given positive constants $\ell_1,\cdots,\ell_d$ and $z_0 \in \mathbb{R}^d$, set
\begin{align*}
Q(\ell_1,\cdots,\ell_d;z_0) = \{ z=(z^1,\cdots,z^d): |z^i - z^i_0| \leq \ell_i, i=1,\cdots,d \}.
\end{align*}

\begin{proof}[Proof of Lemma \ref{Lemma:OperatorSymmDiv}]
	Denote
	\begin{align}
	\bbarDt = \ptl_t + \bar{U}(\Phi_{t-t_0}(t_0,x_0,y_0)) \cdot \nabla_{x,y}.
	\end{align}
	Let $\tilde{\zeta}(x,y)$ be a smooth bump function such that
	\begin{align}
	\supp \tilde{\zeta} \subset  Q(\bar{\ell}_x,\bar{\ell}_x,\bar{\ell}_y;(t_0,x_0,y_0)), \quad
	\int \tilde{\zeta}(x,y)dxdy = 1, \label{supp-zeta'} \\
	\|\nabla_x^{\alpha} \ptl_y^{\beta}    \tilde{\zeta}\|_{C^0} \leq C_{\alpha,\beta}  \bar{\ell}_x^{-|\alpha|} \bar{\ell}_y^{-\beta} |Q(\bar{\ell}_x,\bar{\ell}_x,\bar{\ell}_y)|^{-1}, \quad |\alpha|,\beta \geq 0. \label{ineq:D-tilde-zeta}
	\end{align}
	Let $\zeta(t,x,y)$ be the transport of $\tilde{\zeta}$ by the flow of $\bbarDt$, i.e., $\zeta(t,x,y)$ solves
	\begin{align}
	\bbarDt \zeta(t,x,y) = 0, \quad \zeta(t_0,x,y) = \tilde{\zeta}(x,y). \label{eq:zeta}
	\end{align}
	Note that $\bbarDt(t,\cdot)$ is a constant vector field for any fixed $t$. It follows from \eqref{supp-zeta'} and \eqref{ineq:D-tilde-zeta} that
	\begin{align}
	\supp \zeta(t,\cdot) \subset Q(\bar{\ell}_x,\bar{\ell}_x,\bar{\ell}_y;\Phi_{t-t_0}(t_0,x_0,y_0)), \quad \int \zeta(t,x,y) dxdy = 1, \nonumber \\
	\|\nabla_x^{\alpha} \ptl_y^{\beta} \zeta(t,\cdot) \|_{C^0} \leq  C_{\alpha,\beta}  \bar{\ell}_x^{-|\alpha|} \bar{\ell}_y^{-\beta} |Q(\bar{\ell}_x,\bar{\ell}_x,\bar{\ell}_y)|^{-1}, \quad |\alpha|,\beta \geq 0. \label{est-zeta}
	\end{align}
	The following expression for $R^{kl}[H]$ is a slight modification of those given in \cite[Proposition 11.1]{IsettOh16}. Denote $z=(x,y)$. Let
	\begin{align*}
	R^{kl}[H] = R_0^{kl}[H]+R_1^{kl}[H]+R_2^{kl}[H], \quad \text{ for } k=1,2,3,l=1,2,
	\end{align*}
	where for $j,l = 1,2$,
	\begin{align*}
	R_0^{jl}[H] &= -\frac{3}{2}\int_{0}^{1} \int \zeta(t,\bar{z})\frac{(x^j-\bar{x}^j)}{\sigma}H^l(t,\frac{z-\bar{z}}{\sigma}+\bar{z})\frac{d \bar{z}}{\sigma^3} d \sigma\\
	& \quad -\frac{3}{2}\int_{0}^{1} \int \zeta(t,\bar{z}) \frac{(x^l-\bar{x}^l)}{\sigma}H^j(t,\frac{z-\bar{z}}{\sigma}+\bar{z})\frac{d \bar{z}}{\sigma^3} d \sigma, \\
	R_0^{3l}[H] &= -3\int_{0}^{1}\int \zeta(t,\bar{z}) \frac{y-\bar{y}}{\sigma}H^l(t,\frac{z-\bar{z}}{\sigma}+\bar{z})\frac{d \bar{z}}{\sigma^3} d \sigma,\\
	R^{jl}_1[H] &=  \frac{1}{2}\int_{0}^{1}\int \sum_{k=1}^{2}(\ptl_{x^k}\zeta)(t,\bar{z})\frac{(x^l-\bar{x}^l)(x^k-\bar{x}^k)}{\sigma^2}H^j(t,\frac{z-\bar{z}}{\sigma}+\bar{z})\frac{d \bar{z}}{\sigma^3} d \sigma\\
	&\quad + \frac{1}{2}\int_{0}^{1}\int \sum_{k=1}^{2} (\ptl_{x^k}\zeta)(t,\bar{z})\frac{(x^j-\bar{x}^j)(x^k-\bar{x}^k)}{\sigma^2}H^l(t,\frac{z-\bar{z}}{\sigma}+\bar{z})\frac{d \bar{z}}{\sigma^3} d \sigma,\\
	R^{jl}_2[H] &= -\int_{0}^{1}\int \sum_{k=1}^{2} (\ptl_{x^k}\zeta)(t,\bar{z})\frac{(x^j-\bar{x}^j)(x^l-\bar{x}^l)}{\sigma^2}H^k(t,\frac{z-\bar{z}}{\sigma}+\bar{z})\frac{d \bar{z}}{\sigma^3} d \sigma,\\
	R^{3l}_1[H] &= R^{3l}_2[H] = 0.
	\end{align*}
	It is clear from these definitions that $R^{kl}[H]$ depends linearly on $H$ and $R^{jl}[H]=R^{lj}[H]$ for $j,l=1,2$.  Furthermore, it is easy to verify that $R^{kl} \in C_c^{\infty}(\hat{Q})$.
	It follows from the proof of \cite[Proposition 11.1]{IsettOh16} that $R^{jl}[H]$ is a smooth solution to the divergence equation \eqref{eq:SymmDiv} with the commutating relations
	\begin{align}
	\bbarDt R^{jl}[H] = R^{jl}[\bbarDt H]. \label{eq:bbDtR}
	\end{align}
	The desired esimates follow from \eqref{est-H}, \eqref{est-zeta},  \eqref{est-barU}, \eqref{ineq:bartau} and \eqref{eq:bbDtR} as in the proof of \cite[Proposition 11.1]{IsettOh16}.
\end{proof}

\bigskip

{\bf Acknowledgement.}
The work was initiated as a part of the PhD thesis of the first author written under the supervision of the second author. The authors warmly thank Tao Tao and Liqun Zhang for very valuable discussions. The research are supported in part by Zheng
Ge Ru Foundation, Hong Kong RGC Earmarked Research Grants
CUHK4041/11P, CUHK4048/13P, CUHK-14305315, NSFC/ICG Joint Research Grant N-CUHK 443/14, a Focus Area Grant from The Chinese
University of Hong Kong, and NSFC Grants 11601258.

\end{document}